\documentclass[12pt]{amsart}
 
\usepackage{amsmath, amsfonts, amssymb, amsthm, cite, amscd, verbatim,fullpage}

\numberwithin{equation}{section}

\makeatletter
\newtheorem*{rep@theorem}{\rep@title}
\newcommand{\newreptheorem}[2]{%
\newenvironment{rep#1}[1]{%
 \def\rep@title{#2 \ref{##1}}%
 \begin{rep@theorem}}%
 {\end{rep@theorem}}}
 \makeatother

\theoremstyle{plain}
\newtheorem{theorem}{Theorem}[section]
\newreptheorem{theorem}{Theorem}
\newtheorem{lemma}[theorem]{Lemma}
\newtheorem{corollary}[theorem]{Corollary}

\newtheorem{proposition}[theorem]{Proposition}
\theoremstyle{definition}
\newtheorem{definition}[theorem]{Definition}
\newtheorem{example}[theorem]{Example}
\newtheorem{remark}[theorem]{Remark}

\newcommand{\Gl}{\operatorname{GL}}
\newcommand{\Pgl}{\operatorname{PGL}}
\newcommand{\Sl}{\operatorname{SL}}
\newcommand{\defi}{\stackrel{\text{\tiny def}}{=}}
\newcommand{\degree}{\operatorname{deg}}

\newcommand{\N}{\operatorname{N}}
\newcommand{\An}{\operatorname{An}}
\newcommand{\order}{\operatorname{ord}}
\newcommand{\wt}{\operatorname{wt}}
\newcommand{\Spec}{\operatorname{Spec} }

\begin{document}

\title{Weierstrass points on the Drinfeld modular curve $X_0(\mathfrak{p})$}
\author{Christelle Vincent}
\address{Department of Mathematics, Stanford University, California 94305}
\email{cvincent@stanford.edu}
\thanks{The author thanks the support of an NSERC graduate fellowship}

\begin{abstract}
Consider the Drinfeld modular curve $X_0(\mathfrak{p})$ for $\mathfrak{p}$ a prime ideal of $\mathbb{F}_q[T]$. It was previously known that if $j$ is the $j$-invariant of a Weierstrass point of $X_0(\mathfrak{p})$, then the reduction of $j$ modulo $\mathfrak{p}$ is a supersingular $j$-invariant. In this paper we show the converse: Every supersingular $j$-invariant is the reduction modulo $\mathfrak{p}$ of the $j$-invariant of a Weierstrass point of $X_0(\mathfrak{p})$. 

\end{abstract}

\maketitle

\section{Introduction and Statement of Results}

Given a smooth irreducible projective curve of genus $g\geq 2$ defined over an algebraically closed field of characteristic $0$, we say that a point $P$ on $X$ is a \emph{Weierstrass point} if there is a nonzero rational function $F$ on $X$ with a pole of order less than or equal to $g$ at $P$ and regular everywhere else. In this case, the set of such points is non-empty and finite.

Because of the geometric significance of such points, given a curve of arithmetic import it is natural to study its Weierstrass points. Such work was done by Atkin, Hasse, Lehner and Newman, Ogg, Petersson, and Schoeneberg for three families that are important to number theorists: the Fermat curves, and the modular curves $X(N)$ and $X_0(N)$. The interested reader should see Rohrlich's 1982 paper \cite{rohrlich1} for a concise account of the results and complete references. In the same paper, Rohrlich exhibited a modular form $W(z)$ for $\Gamma_0(N)$ whose divisor encodes information about the Weierstrass points of $X_0(N)$, the modular Wronskian. In later work \cite{rohrlich2}, restricting his attention to $\Gamma_0(\ell)$ for $\ell$ a prime, he was able to exhibit a form for $\Sl_2(\mathbb{Z})$ congruent to $W(z)$ modulo $\ell$. Building on these results, later work of Ahlgren and Ono \cite{ahlgrenono} showed that not only were the elliptic curves underlying the Weierstrass points of $X_0(\ell)$ supersingular at $\ell$, which was a result already obtained by Ogg \cite{ogg}, but furthermore that 
\begin{equation*}
\prod_{Q\in X_0(\ell)} (x-j(Q))^{\wt(Q)} \equiv \prod_{\substack{E/\overline{\mathbb{F}}_{\ell} \\ E \text{ supersingular}}}(x-j(E))^{g_{\ell}(g_{\ell}-1)} \pmod{\ell},
\end{equation*} 
where the quantity $\wt(Q)$ is a non-negative integer which is positive if and only if $Q$ is a Weierstrass point and which we will define in Section \ref{weierstrass}, Definition \ref{D:weight}, and $g_{\ell}$ is the genus of $X_0(\ell)$.


The situation where the curve is defined over an algebraically closed field of positive characteristic is more complicated: It can be the case that for each point $P$ there exists a nonzero rational function with a pole of order less than or equal to the genus of the curve at $P$ and regular elsewhere. Accordingly, to ensure that the set of Weierstrass points be finite, a modified definition of Weierstrass points must be used, which will be given in Section \ref{theory}.

We consider in this paper the so-called Drinfeld setting, which offers for function fields some structures playing roles analogous to those played by elliptic curves, modular forms and modular curves for number fields. More precisely, we will study the Weierstrass points on a family of Drinfeld modular curves which is denoted by $X_0(\mathfrak{p})$, where $\mathfrak{p}$ is a prime ideal of $\mathbb{F}_q[T]$. These curves are smooth, irreducible and projective, and defined over a complete, algebraically closed field of positive characteristic. As such, it is natural to wish to study their Weierstrass points. Since they are (coarse) moduli spaces of Drinfeld modules of rank $2$ with a specified level structure, we may ask what can be said about the Drinfeld modules underlying the Weierstrass points of $X_0(\mathfrak{p})$.



As far as we can tell, the only result in this direction which was known previously was obtained by Baker \cite{baker} as a result of his work on the connection between linear systems on a curve and linear systems on the dual graph of a regular semistable model of the curve. As a corollary of one of his results, one can show that the Drinfeld modules underlying the Weierstrass points of $X_0(\mathfrak{p})$ have supersingular reduction at $\mathfrak{p}$. 

In this paper we prove a converse of Baker's result:

\begin{theorem}\label{T:main}
Let $q$ be odd and let $\pi(T) \in \mathbb{F}_q[T]$ be a prime polynomial, generating the prime ideal $\mathfrak{p}$. Then each supersingular Drinfeld module over $\overline{\mathbb{F}}_{\mathfrak{p}}$ is the reduction modulo $\mathfrak{p}$ of a Weierstrass point of $X_0(\mathfrak{p})$.
\end{theorem}

To obtain this theorem we first introduce the necessary concepts and objects to define a form $W(z)$ analogous to the form defined by Rohrlich in \cite{rohrlich1}. By this we mean that the divisor of $W(z)$ captures information about the Weierstrass points of $X_0(\mathfrak{p})$, and the $u$-series coefficients of $W(z)$ at the cusp $\infty$ are rational and $\mathfrak{p}$-integral. It is the study of this form, using the main theorems of \cite{normandtrace}, that allows us to use a powerful theorem on the arithmetic of the reduction of Drinfeld modular forms modulo a prime ideal $\mathfrak{p}$ and obtain Theorem \ref{T:main}. 

\begin{remark}
The hypothesis that $q$ be odd in our main theorem is a consequence of Theorem \ref{T:qodd}, in which we show that $W(z)$ is an eigenform of the Fricke involution in odd characteristic. In turn, this hypothesis is necessary to apply one of the theorems of \cite{normandtrace} (repeated here as Theorem \ref{normtheorem}). We expect that $W(z)$ is an eigenform of the Fricke involution in even characteristic as well, but the argument in this case would most likely rely on some geometric property of $W(z)$ on $X_0(\mathfrak{p})$ instead of the more ``modular" argument we present here. Granting this hypothesis, the proof of Theorem \ref{T:main} would carry through for $q>2$. The exclusion of the case $q=2$ would now come from the other main theorem of \cite{normandtrace} (repeated here as Theorem \ref{tracetheorem}); see the remark following the proof of Theorem 1.1 in \cite{normandtrace} for a discussion of how this restriction arises.
\end{remark}

In Section \ref{specialcase}, we show that when $q=p$ is an odd prime and $\pi(T)$ has degree 3, we can perform some explicit computations to obtain that

\begin{theorem}\label{T:computation}
If $p$ is odd, $\pi(T) \in \mathbb{F}_p[T]$ has degree 3, $\mathfrak{p}$ is the ideal generated by $\pi(T)$, and the modular Wronskian on $X_0(\mathfrak{p})$ is denoted by $W(z)$, then
\begin{equation*}
W(z)\equiv (-1)^{(p+1)/2}g^{\frac{p^2(p-1)}{2}}h^{\frac{p^2(p+1)}{2}} \pmod{\mathfrak{p}}.
\end{equation*}
\end{theorem}

Here $g$ and $h$ are explicit Drinfeld modular forms which will be defined in Section \ref{drinfeldthry}. Theorem \ref{T:computation} is an analogue of a result obtained by Rorhlich in \cite{rohrlich2}. This explicit computation allows us to show that 

\begin{theorem}\label{T:ahlgrenono}
If $p$ is odd, $\pi(T) \in \mathbb{F}_p[T]$ has degree 3, and $\mathfrak{p}$ is the ideal generated by $\pi(T)$, then we have
\begin{equation*}
\prod_{P \in Y_0(\mathfrak{p})} (x- j(P))^{\wt(P)} \equiv \prod_{\substack{\phi/\overline{\mathbb{F}}_{\mathfrak{p}} \\ \phi\text{ supersingular}}}(x-j(\phi))^{g_{\mathfrak{p}}(g_{\mathfrak{p}}-1)} \pmod{\mathfrak{p}},
\end{equation*}
where $g_{\mathfrak{p}}$ is the genus of the curve $X_0(\mathfrak{p})$ and $\wt(P)$ is given in Definition \ref{D:weight}.
\end{theorem}

This is an analogue of the formula from \cite{ahlgrenono} quoted earlier.

The structure of the paper is the following: We begin by reviewing the theory of Weierstrass points in positive characteristic in Section \ref{theory}. Then we introduce the basic objects from the Drinfeld setting that we will need in Section \ref{drinfeldthry}. Section \ref{drinfeldthry} contains as well all of the statements of the results from the theory of Drinfeld modular forms that we will cite. In Section \ref{quasimodular}, we introduce Drinfeld quasimodular forms and some differentials operators that are needed in the definition of the Drinfeld modular form $W(z)$. We also prove some elementary results concerning the action of these operators on Drinfeld modular forms. The definition of $W(z)$ is finally given in Section \ref{weierstrass}, along with the properties of this form. Then the meat of the proof of Theorem \ref{T:main} is in Section \ref{consequences} where we apply the machinery developed in the previous sections to the form $W(z)$. Finally, in Section \ref{divisorcusp} we briefly consider the order of vanishing of $W(z)$ at $\infty$ and establish a result needed to study the special case which yields Theorems \ref{T:computation} and \ref{T:ahlgrenono}. The proofs of these last two theorems are then given in Section \ref{specialcase}.

\section*{Acknowledgments}
A lot of this work is part of the author's PhD thesis, and the author is grateful to her adviser Ken Ono, who suggested the problem. The author also thanks Matt Baker for many helpful reading suggestions, and C\'{e}cile Armana for showing her how the computation of the canonical gap sequence of $X_0(\mathfrak{p})$ follows from one of her results. Finally, the author thanks the referee for their thoughtful comments and suggestions.

\section{Weierstrass points in characteristic $p$}\label{theory}

Since the theory of Weierstrass points in positive characteristic $p$ is much less well known than the theory in characteristic $0$, we begin with a short review of the facts we will need, based on the treatment in \cite{stohrvoloch} and \cite{goldschmidt}. In particular, proofs of all facts that are stated here without proof can be found in \cite{goldschmidt}.

For the duration of this section only, let $k$ be an algebraically closed field and $X$ a smooth projective irreducible curve over $k$ of genus $g\geq 2$ with function field $k(X)$. A natural question to ask about $X$ is the following: For $P$ a point of $X$ and $n$ a positive integer, does there exist a nonzero rational function $F$ on $X$ such that $F$ has a pole of order exactly $n$ at $P$ and $F$ is regular elsewhere? If the answer to this question is negative, we say that $n$ is a \emph{gap} at $P$; otherwise $n$ is a \emph{pole number} at $P$. It is a fact that for a point $P$ on $X$ there are exactly $g$ gaps at $P$, and if $n_1(P), \ldots, n_g(P)$ are the gaps at $P$, indexed such that  $n_i (P)<n_j(P)$ if $i<j$, we say that $(n_1(P), \ldots, n_g(P))$ is the gap sequence at $P$. 

For a fixed curve $X$, it can be shown that there exists a sequence of positive integers $(n_1, \ldots, n_g)$ with $n_i <n_j$ if $i<j$ such that $(n_1, \ldots, n_g)$ is the gap sequence for all but finitely many points of $X$. We call this sequence the \emph{canonical gap sequence} of $X$. There exist on $X$ finitely many points that have a different gap sequence, and they are called the \emph{Weierstrass points} of $X$. If $(n_1, \ldots, n_g)$ is the canonical gap sequence of $X$ and $(n_1(P), \ldots, n_g(P))$ is the gap sequence at $P$ for any point $P$ of $X$, then $n_i \leq n_i(P)$ for each $i$.

\begin{remark}\label{genusconstraint}
We exclude the case of $g=0$ since in that case for any point $P$ there is a nonzero rational function $F$ on $X$ such that $F$ has a single pole at $P$ and is regular elsewhere.  There are therefore no Weierstrass points. We also exclude the case of $g=1$ since in that case there are no points $P$ with a pole number of $1$ (the existence of such a point would force $g=0$, a contradiction) and for each $P$ on $X$ there is a nonzero rational function $F$ on $X$ such that $F$ has a double pole at $P$ and is regular elsewhere.  There are therefore again no Weierstrass points. 
\end{remark}

For any point $P$ on $X$, a measure of how its gap sequence differs from the canonical gap sequence is given by the quantity
\begin{equation*}
\sum_{i=1}^{g} (n_i(P)-n_i),
\end{equation*}
which is positive if and only if $P$ is a Weierstrass point.

If $X$ is defined over a field of characteristic $0$, then the canonical gap sequence is always $(1, \ldots, g)$. When $k$ is of characteristic $p>0$ and $X$ has canonical gap sequence $(1, \ldots, g)$, we say that $X$ has a \emph{classical} canonical gap sequence, or a classical canonical linear system (this designation will be justified shortly when we define the canonical orders of $X$).

\begin{example}
Let $X$ be a hyperelliptic curve of genus $g \geq 2$, then its canonical gap sequence is $(1, \ldots, g)$. (In characteristic $p>0$ this is a theorem that was implicit in \cite{hasseschmid} and stated explicitly in the seminal work of Schmidt \cite{schmidt} defining Weierstrass points in positive characteristic.) Furthermore, the Weierstrass points of $X$ are exactly the branch points of $f$, where $f \colon X \to \mathbb{P}^1$ is any separable degree $2$ morphism. At such a branch point $P$ the rational function $F=\frac{1}{f-f(P)}$ has a double pole at $P$ and is regular elsewhere, and so at the Weierstrass points the gap sequence is $(1,3, \ldots, 2g-1)$.
\end{example}

\begin{example}
The projective curve of genus $3$ given by $X_0^4+X_1^4+X_2^4=0$ over $\overline{\mathbb{F}}_3$ does not have a classical gap sequence. On this curve, for each point $P$ one can construct a nonzero rational function having a pole of order $\leq 3$ at $P$ and regular elsewhere. 
\end{example}

Because of the difficulty of computing the gap sequence of a point directly, it is often more convenient to consider a related sequence of strictly increasing positive integers $(j_1, \ldots, j_{g-1})$ called the \emph{canonical orders} of $X$, which we now describe. For any element $x \in k(X)$, we will write $[x]$ for the divisor of $x$, $\sum_{P} v_{P}(x)P$, where the sum is taken over all points $P$ of $X$. As usual, for any divisor $D$ on $X$, we may define the linear system
\begin{equation*}
L(D)=\{ x \in k(X)^{\times} : [x] \geq -D\} \cup \{0\}.
\end{equation*}

We further denote by $\Omega_X$ the space of (algebraic) meromorphic differential forms on $X$. Because $X$ is defined over an algebraically closed field, we have a canonical isomorphism between $\Omega_X$ and the space of Weil differentials $W_X$ (in fact, to obtain this isomorphism it would suffice here to require that $k' \otimes k(X)$ be a field for all finite extensions $k'$ of $k$). This allows us to define the divisor $[\omega]$ of $\omega$ a meromorphic differential on $X$. We do this in the following manner: Let $\mathbb{A}_{k(X)}$ denote the ring of ad\`{e}les of $k(X)$ and for $D$ a divisor on $X$, write
\begin{equation*}
\mathbb{A}_{k(X)}(D) \defi \{ \alpha = (\alpha_P) \in \mathbb{A}_{k(X)} \mid v_{P}(\alpha_P) \geq - v_{P}(D) \text{ for all points $P$ of $X$} \}.
\end{equation*}
Then a Weil differential on $X$ is a $k$-linear functional with domain $\mathbb{A}_{k(X)}$ that vanishes on $\mathbb{A}_{k(X)}(D) + k(X)$ for some divisor $D$. For each Weil differential $\omega^*$, there is a unique divisor $D$ of maximum degree such that  $\omega^*$ vanishes on $\mathbb{A}_{k(X)}(D)+k(X)$, and we define $[\omega^*] \defi D$. Then if $\omega$ corresponds to $\omega^*$ under the canonical isomorphism between Weil differentials and meromorphic differentials, we simply write $[\omega]=[\omega^*]$ and $v_P(\omega)=v_P([\omega])$. One pleasant consequence of this definition is that for $x\in k(X)$ and $\omega \in \Omega_X$ we have $[x\omega]=[x]+[\omega]$. If $\omega$ is a meromorphic differential on $X$, its divisor $C$ is called a \emph{canonical divisor} on $X$, and since any two meromorphic differentials differ by a function, any two canonical divisors are linearly equivalent.

For a point $P$ of $X$, consider the following sequence of spaces:
\begin{equation*}
k= L(0) \subseteq L(P) \subseteq L(2P) \subseteq L(3P) \subseteq \ldots
\end{equation*}
Then we have that $n$ is a gap at $P$ if and only if $L((n-1)P)=L(nP)$. By the Riemann-Roch theorem, we have that for any positive integer $n$ and any point $P$, 
\begin{equation*}
\dim L(nP)=n-g+1+\dim L(C-nP),
\end{equation*}
from which it follows that 
\begin{equation*}
\dim L((n+1)P)/L(nP) = 1 - \dim L(C-nP)/L(C-(n+1)P).
\end{equation*}
Writing $L_C(nP)=L(C-nP)$, this last equation justifies our interest in the \emph{(canonical) osculating filtration} at $P$:
\begin{equation*}
L(C)=L_C(0) \supseteq L_C(P) \supseteq L_C(2P) \supseteq L_C(3P) \supseteq \ldots
\end{equation*}
Indeed, for a positive integer $n$, $n+1$ is a gap at $P$ if and only if $L_C(nP) \supsetneq L_C((n+1)P)$. In turn, this implies the existence of a nonzero $F \in L(C)$ such that $v_P(F)=n-v_P(C)$. Whenever such a function exists, we say that $n$ is a \emph{canonical order} at $P$. The definition of the canonical orders at $P$ does not depend on the choice of canonical divisor $C$: if $n$ is a canonical order at $P$ and $C'$ is any canonical divisor, there will exist a nonzero $F' \in L(C')$ such that $v_P(F')=n-v_P(C')$.

From the discussion above it follows that for a positive integer $n$, $n$ is a canonical order at $P$ if and only if $n+1$ is a gap at $P$. (We note that since $X$ is a curve over an algebraically closed field, the existence of a point $P$ such that $1$ is a pole number at $P$ implies that $X$ has genus zero. Therefore in our case $1$ will always be a gap for any point $P$ on $X$ since we restrict our attention to curves of genus greater than or equal to $2$, but we do not say that $0$ is a canonical order.) As was the case for gap sequences, all but finitely many points of $X$ have the same canonical orders, and we call the strictly increasing sequence of positive integers $(j_1, \ldots, j_{g-1})$ formed by these integers the \emph{canonical orders} of $X$.

If $(j_1, \ldots, j_{g-1})$ are the canonical orders of $X$ and $(j_1(P), \ldots, j_{g-1}(P))$ are the canonical orders at $P$ for any point $P$ of $X$, then again $j_i \leq j_i(P)$ for each $i$. Furthermore if as before $(n_1, \ldots, n_g)$ is the canonical gap sequence of $X$ and $(n_1(P), \ldots, n_g(P))$ is the gap sequence at $P$ then
\begin{equation*}
\sum_{i=1}^{g} (n_i(P)-n_i)=\sum_{i=1}^{g-1} (j_i(P)-j_i).
\end{equation*}

The point $P$ is called an \emph{osculation point} of $X$ if $j_{g-1}(P) > g-1$. In particular an osculation point has at least one pole number that is less than or equal to $g$. If $X$ has a classical gap sequence, then the osculation points and the Weierstrass points of $X$ exactly coincide. Otherwise, every point of $X$ is an osculation point. 

An important tool in the study of Weierstrass points is a divisor $w$ on $X$, whose construction is due to St\"{o}hr and Voloch \cite{stohrvoloch}. This divisor has the property that 
\begin{equation*}
v_P(w) \geq \sum_{i=1}^{g} (n_i(P)-n_i)
\end{equation*}
for any point $P$ of $X$, with equality 
\begin{equation*}
v_P(w)=\sum_{i=1}^{g} (n_i(P)-n_i)=0
\end{equation*}
if $P$ is not a Weierstrass point of $X$. We describe its construction now.

A \emph{separating variable} for $k(X)$ is an element $s \in k(X)$ transcendental over $k$ such that $k(X)$ is a finite, separable extension of $k(s)$. With the assumptions on $X$ enforced in this section, we have that $s$ is a separating variable if and only if the differential $ds$ is not identically $0$. Furthermore, $s$ is a separating variable if $s$ is a local parameter at a separable point of $X$. Since in our case $X$ is defined over an algebraically closed field $k$, every point is separable.

On the polynomial ring $k[s]$, we may define the \emph{$n^{th}$ Hasse derivative with respect to $s$} by putting
\begin{equation*}
\mathfrak{D}_s^{(n)}(s^m)= 
\begin{cases}
\binom{m}{n} s^{m-n} & \text{if $m \geq n$,}\\
0 & \text{otherwise,}
\end{cases}
\end{equation*}
and extending linearly to $k[s]$. It can be shown that if $s$ is a separating variable for $k(X)$ over $k$, then this family of maps can be uniquely extended to a family of maps $\mathfrak{D}_s^{(n)}: k(X) \to k(X)$.

Again let $C$ be a canonical divisor on the curve $X$, and consider the linear system $L(C)$ associated to it. It is a basic fact that $L(C)$ is a $k$-vector subspace of $k(X)$ of dimension $g$, and that replacing $C$ by a different canonical divisor yields an isomorphic subspace. Fix any basis $\phi=\{\phi_1, \ldots \phi_g\}$ of $L(C)$, and define the matrix
\begin{equation*}
 H= H(\phi,s)=\left(\mathfrak{D}_s^{(j)}(\phi_i)\right)
 \end{equation*}
 for $1 \leq i \leq g$ and $0 \leq j$. Write further $H^{(j)}$ for the column of $H$ whose $i^{th}$ entry is $\mathfrak{D}_s^{(j)}(\phi_i)$.
 
We are interested in the indices $j$ such that $H^{(j)}$ is not a $k(X)$-linear combination of lower-numbered columns. This is true for $j=0$ since the $\phi_i$'s are not all zero. One can show that there are $g-1$ more such indices, which we will denote by $j_1, \ldots, j_{g-1}$, and we will write $J(\phi, s)=(j_1, \ldots , j_{g-1})$.This sequence has the property that $J(\phi, s)$ in fact does not depend on our choice of $s$ a separating variable, $C$ a canonical divisor, or $\phi$ a basis for the linear system associated to $C$, and in fact that the $j_i$'s are exactly the canonical orders of $X$ defined earlier.
 
For any sequence $J=(j_1, j_2, \ldots)$ of positive integers, let $H^J$ be the submatrix of $H$ whose first column is $H^{(0)}$ and whose $(l+1)^{st}$ column is $H^{(j_l)}$. Then we may define the nonzero rational function
 \begin{equation*}
 W(\phi, s) = \det H^{J(\phi, s)},
 \end{equation*} 
the \emph{Wronskian of $\phi$ with respect to $s$}. While not independent of the choices made above, this function behaves as well as well as can be expected. More precisely, put $\phi'_i=\sum_j a_{ij}\phi_j$ for $a_{ij} \in k$ such that $\phi'=(\phi_1', \ldots, \phi_g')$ is a different basis for $L(C)$, and let $y \in k(X)^{\times}$ and $t$ be another separating variable. Then
\begin{equation}\label{wronsk}
W(y\phi', t)=\det(a_{ij})y^g\left(ds/dt\right)^{j_1+ \ldots +j_{g-1}}W(\phi, s).
\end{equation}
In light of this equation, we define the following divisor:
\begin{equation*}
w(\phi, s) = [W(\phi, s)]+gC+(j_1+\ldots+j_{g-1})[ds],
\end{equation*}
which by equation (\ref{wronsk}) is in fact independent of any choice we made, so that we may denote it simply by $w$. One can show that the points in the support of $w$ are exactly the Weierstrass points of $X$, and that $v_P(w) \geq \sum_{i=1}^{g} (n_i(P)-n_i)$ for any point $P$ of $X$, with equality $v_P(w)=\sum_{i=1}^{g} (n_i(P)-n_i)=0$ if $P$ is not a Weierstrass point of $X$, as claimed above.

The divisor $w$ is effective: Fixing a point $P$ of $X$, one may choose a canonical divisor $C$ such that $v_P(C)= 0$, which ensures that $v_P(\phi_i)\geq 0$, so that $v_p([W(\phi, s)]) \geq 0$ since taking Hasse derivatives does not lower the valuation. Furthermore, one can choose $s$ to be a local parameter at $P$, so that $v_P([ds])=0$. With these choices and because of the invariance of $w$, it follows that $v_P(w)\geq 0$ for each $P$.

In \cite{stohrvoloch}, the authors define the Weierstrass weight of a point to be $v_P(w)$. In Section \ref{weierstrass} we will define a Drinfeld modular form $W(z)$ that will play for us a role analogous to the function $W(\phi,s)$. Because of this analogy, we will use the divisor of $W(z)$ to define the modular Weierstrass weight of a point $P$ on the Drinfeld modular curve $X_0(\mathfrak{p})$, and study this integer in this work.

\begin{remark}
If $X$ is defined over a field of characteristic $0$, we have the equality $v_P(w) = \sum_{i=1}^{g} (n_i(P)-n_i)$ for all points $P$ of $X$. In positive characteristic, this equality holds if and only if $\det \binom{J'}{J} \neq 0$, where $J'=(j_1(P), \ldots, j_{g-1}(P))$ is the sequence of canonical orders at $P$, $J=(j_1, \ldots, j_{g-1})$ is the sequence of canonical orders of $X$, and $\binom{J'}{J}$ is the $(g-1) \times (g-1)$ matrix of binomial coefficients $\binom{j'_r}{j_s}$, where $\binom{j'_r}{j_s}=0$ if $j'_r<j_s$ and each binomial coefficient is reduced modulo $p$, the characteristic of $k$.
\end{remark}

We will also need the following well-known fact: We have that $\dim_{k(X)} \Omega_X=1$, so that $\Omega_X=k(X)\cdot \omega$ for any non-zero $\omega \in \Omega_X$. If $C$ is a canonical divisor of $X$, by definition it is the divisor of some Weil differential $\omega^*$ and thus of a meromorphic differential $\omega$. Then the map
\begin{align*}
\Omega_X &\to k(X)\\
x\omega &\mapsto x
\end{align*}
is an isomorphism of $k$-vector spaces, and under this isomorphism the space $L(C) \subset k(X)$ corresponds to the space $\Omega_{X, \text{reg}}$ of algebraic differentials without poles.

\section{The Drinfeld setting}\label{drinfeldthry}

Throughout when we refer to rigid analytic objects we will mean rigid analytic in the sense of Fresnel and van der Put \cite{fresnel}.

\subsection{Drinfeld modules and Drinfeld modular forms}
For a reference on Drinfeld modules and Drinfeld modular forms, we refer the reader to Gekeler's excellent \emph{Inventiones} paper \cite{gekeler}, or to the author's PhD thesis \cite{vincentthesis}. 

In this paper we will only consider the case of the affine ring $A=\mathbb{F}_q[T]$, with fraction field $K=\mathbb{F}_q(T)$. We complete $K$ at the infinite place $v_{\infty}(x)=-\operatorname{deg}(x)$, and write $K_{\infty}=\mathbb{F}_q(\!(1/T)\!)$ for the completion of $K$ at this place. We will also write
\begin{equation*}
C=\hat{\bar{K}}_{\infty}
\end{equation*}
for the completed algebraic closure of $K_{\infty}$, and $\Omega=\mathbb{P}^1(C)-\mathbb{P}^1(K_{\infty})=C - K_{\infty}$. (From now on $C$ will never be a canonical divisor again.) $\Omega$ has a rigid analytic structure described in \cite{gekelerjacobian}, and we call it the Drinfeld upper half-plane. The group $\Gl_2(A)$ acts on $\Omega$ by fractional linear transformations.

\begin{definition}\label{modform}
Let $\Gamma$ be a congruence subgroup of $\Gl_2(A)$. A function $f \colon \Omega \rightarrow C$ is called a \emph{Drinfeld modular form of weight $k$ and type $l$ for $\Gamma$}, where $k \geq 0$ is an integer and $l$ is a class in $\mathbb{Z} / (\# \det \Gamma)$, if
\begin{enumerate}
\item for $\gamma = 
\left(\begin{smallmatrix} a&b\\ c&d \end{smallmatrix}\right)
\in \Gamma$, $f(\gamma z)=(\det \gamma)^{-l}(cz+d)^kf(z)$;
\item $f$ is rigid analytic on $\Omega$;
\item \label{cusps} $f$ is analytic at the cusps of $\Gamma$: at each cusp $f$ can be written as a power series with a positive radius of convergence in a (root of) a local parameter at this cusp (this will be discussed further shortly).
\end{enumerate}
\end{definition}

For a congruence subgroup $\Gamma$ of $\Gl_2(A)$, we will denote the (finite dimensional) vector space of Drinfeld modular forms of weight $k$ and type $l$ for this subgroup by $M_{k,l}(\Gamma)$, the subspace of cusp forms (the forms having at least a single zero at each cusp of $\Gamma$) by $M^{1}_{k,l}(\Gamma)$, and the subspace of double cusp forms (the forms having at least a double zero at each cusp of $\Gamma$) by $M^{2}_{k,l}(\Gamma)$. We will define precisely what we mean by the order of vanishing of a Drinfeld modular form at a cusp at the very end of this section.

Although they are important to this work, we will avoid discussing Drinfeld modules as much as possible, referring rather the reader to \cite{gekeler} for background reading. We limit ourselves to defining the Carlitz module, and presenting only the barest facts about Drinfeld modules of rank $2$ that are necessary to read the text. 

\begin{definition}
Let $L$ be either a field extension of $K$ or, if $\mathfrak{p}$ is a prime ideal of $A$, an extension of the field $\mathbb{F}_{\mathfrak{p}} = A/\mathfrak{p}$. Further write $\tau(X)=X^q$ and let $L\{\tau\} \subset \operatorname{End}_L(\mathbb{G}_a)$ be the subalgebra generated by $\tau$ over $L$, with commutation relation $l \tau = \tau l^q$ for $l \in L$. A Drinfeld module of rank $r$ over $L$ is a ring homomorphism $\phi \colon A \to L \{ \tau \}$ such that for $a \in A$ of degree $d$,
\begin{equation*}
\phi(a) = \sum_{0\leq i\leq rd}l_i\tau^i
\end{equation*}
with $l_0=a$ and $l_{rd}\neq 0$. The numbers $l_i$ are called the \emph{coefficients} of $\phi$.
\end{definition}

We say that two Drinfeld modules $\phi$ and $\psi$ are \emph{isogenous} if there exists a nonzero element $u \in \operatorname{End}_L(\mathbb{G}_a)$ such that $u \circ \phi = \psi \circ u$. If $u\in L^{\times}$, then we say that the two modules are \emph{isomorphic} over $L$. One may show that there is a one-to-one correspondence between Drinfeld modules of rank $r$ over $C$ and rank $r$ $A$-lattices in $C$.

In the 1930s Carlitz studied some polynomials which had properties similar to those exhibited by the classical cyclotomic polynomials \cite{carlitz4}. Reinterpreting his work in the context laid out by Drinfeld, his polynomials are now understood to give the action on $C$ of a certain Drinfeld module of rank $1$. We call this module the Carlitz module and it is defined by:
\begin{equation}\label{carlitz}
\rho(T)=T\tau^0+\tau.
\end{equation}
Under the correspondence mentioned above, this Drinfeld module corresponds to a certain rank 1 $A$-lattice $L=\tilde{\pi}A$, where the \emph{Carlitz period} $\tilde{\pi}\in K_{\infty}(\sqrt[q-1]{-T})$ is defined up to multiplication by an element of $\mathbb{F}_q^{\times}$. We choose one such $\tilde{\pi}$ and fix it for the remainder of this work. As usual we have the Carlitz exponential function
\begin{equation*}\label{exponential}
e_{A}(z) \defi z\prod_{\substack{a \in A\\ a \neq 0}}\left(1-\frac{z}{a}\right).
\end{equation*}
Then we write
\begin{equation}\label{u}
u(z) \defi \tilde{\pi} \frac{1}{e_A(z)}
\end{equation} 
for the parameter at infinity. This differs from Gekeler's original notation, who used $t(z)$ for this function, but agrees with the notation used in more recent articles, for example by Bosser and Pellarin in \cite{bosserp2}.

We will also consider Drinfeld modules of rank $2$. For $a \in A$, $\phi$ a Drinfeld module over $L$, and $L'$ a field extension of $L$, write
\begin{equation*}
\phi[a](L')= \{ x \in L' : \phi(a)(x)=0 \}
\end{equation*}
for the $a$-torsion of $\phi$. When $\phi$ is of rank $2$ and defined over $C$, for $\pi(T)$ a prime polynomial generating the ideal $\mathfrak{p}$ of $A$, we have
\begin{equation*}
\phi[\pi](C) \cong A/\mathfrak{p} \times A/\mathfrak{p}.
\end{equation*}

Again, if $\phi$ is of rank $2$, but is now defined over the algebraic closure $\overline{\mathbb{F}}_{\mathfrak{p}}$ of $\mathbb{F}_{\mathfrak{p}} = A/\mathfrak{p}$, we have
\begin{equation*}
\phi[\pi](\overline{\mathbb{F}}_{\mathfrak{p}})=
\begin{cases}
0 \qquad & \text{in which case we say $\phi$ is \emph{supersingular}, or}\\
A/\mathfrak{p} \qquad &\text{in which case we say $\phi$ is \emph{ordinary}}.
\end{cases}
\end{equation*}
There are $g_{\mathfrak{p}}+1$ supersingular Drinfeld modules defined over the algebraic closure of $A/\mathfrak{p}$, where
\begin{equation}\label{genusformula}
g_{\mathfrak{p}} \defi
\begin{cases}
\frac{q^{d}-q}{q^2-1} & \qquad \mbox{if $d$ is odd,}\\
\frac{q^{d}-q^2}{q^2-1} & \qquad \mbox{if $d$ is even.}
\end{cases}
\end{equation}

\begin{remark}
We use $g_{\mathfrak{p}}$ to denote the quantity above because it is the genus of the modular curve $X_0(\mathfrak{p})$.
\end{remark}

As before, let $\mathfrak{p}$ be a prime ideal of $A$. For a Drinfeld module $\phi$ of rank $2$ over $K$, there is a notion of good reduction at $\mathfrak{p}$: First one must find a Drinfeld module $\psi$ isomorphic to $\phi$ over $K$, such that $\psi$ has coefficients in $A$ and such that the reduction of $\psi$ modulo $\mathfrak{p}$ (obtained by reducing the coefficients modulo $\mathfrak{p}$) is a Drinfeld module. If this is possible and in addition the reduction of $\psi$ modulo $\mathfrak{p}$ has rank $2$ as a Drinfeld module, then we say that $\phi$ has \emph{good reduction at $\mathfrak{p}$}. Furthermore, if the reduction of $\psi$ modulo $\mathfrak{p}$ is supersingular, then we say that $\phi$ is \emph{supersingular at $\mathfrak{p}$}.

We now present some facts on Drinfeld modular forms for the full modular group $\Gl_2(A)$. As in the classical case, the algebraic curve $Y_{\Gl_2(A)}$ whose associated rigid analytic space is $\Gl_2(A)\backslash \Omega$ can be compactified by adding a single cusp which we denote by $\infty$. This will be discussed more rigorously in the next section.

As in \cite{gekeler}, we will write $g_k$ for the normalized Eisenstein series of weight $q^k-1$ and type $0$ for $\Gl_2(A)$ and set $g=g_1$ for simplicity. (From now on we will never use $g$ to denote the genus of a curve again.) We will also write $h$ for the Poincar\'{e} series of weight $q+1$ and type $1$ for $\Gl_2(A)$ which was first defined in \cite{gerritzen}. It is well-known that the graded $C$-algebra of Drinfeld modular forms of all weights and all types for $\Gl_2(A)$ is the polynomial ring $C[g,h]$ (where each Drinfeld modular form corresponds to a unique isobaric polynomial), that $g$ has leading term $1$, that $h$ has a single zero at $\infty$ and leading coefficient $-1$ and that both $g$ and $h$ have  $u$-series expansions with integral coefficients.

We record here a computation which we will need later, and which follows from knowing that the algebra of Drinfeld modular forms is generated by $g$ and $h$:

\begin{proposition}\label{dimensiongl}
For $q\geq 3$, the dimension of the space of modular forms of weight $q^d+1$ and type $1$ for $\Gl_2(A)$ is equal to $g_{\mathfrak{p}}+1$, and the dimension of its subspace of double cusp forms is $g_{\mathfrak{p}}$, where $g_{\mathfrak{p}}$ is as in equation (\ref{genusformula}).
\end{proposition}


We will also need a slash operator, which we define now. For any $x \in K_{\infty}^{\times}$, $x$ can be written uniquely as 
\begin{equation}\label{leadingcoeff}
x=\zeta_x \left(\frac{1}{T}\right)^{v_{\infty}(x)}u_x
\end{equation}
where $\zeta_x \in \mathbb{F}_q^{\times}$, and $u_x$ is such that $v_{\infty}(u_x-1)>0$, or in other words $u_x$ is a $1$-unit at $\infty$. We call $\zeta_x$ the \emph{leading coefficient} of $x$.

For $\gamma \in \Gl_2(K)$ we have that $\det \gamma \in K^{\times}$. By (\ref{leadingcoeff}), we can write
\begin{equation*}
\det \gamma= \zeta_{\det \gamma} \left(\frac{1}{T}\right)^{v_{\infty}(\det \gamma)}u_{\det \gamma}.
\end{equation*}
For simplicity we write 
\begin{equation*}
\zeta_{\det \gamma}=\zeta_{\gamma}.
\end{equation*}

We define a \emph{slash operator} for $\gamma = 
\left(\begin{smallmatrix} a&b\\ c&d \end{smallmatrix}\right)
\in \Gl_2(K)$ on a modular form of weight $k$ and type $l$ by
\begin{equation}\label{slash}
f|_{k,l}[\gamma]=\zeta_{\gamma}^l \left(\frac{\det \gamma}{\zeta_{\gamma}}\right)^{k/2}(cz+d)^{-k}f(\gamma z).
\end{equation}
Note that for $\gamma \in \Gl_2(A)$ we have that $\det \gamma =\zeta_{\gamma}$; thus if $f$ is modular of weight $k$ and type $l$ for $\Gamma$ and $\gamma \in \Gamma$, then $f|_{k,l}[\gamma]= f$. 

\subsection{Drinfeld modular forms modulo $\mathfrak{p}$}

An important tool we will use to study the Weierstrass points of the curve $X_0(\mathfrak{p})$ is the theory of Drinfeld modular forms for $\Gl_2(A)$ upon reduction modulo $\mathfrak{p}$. Everywhere in this paper we will write $\pi(T) \in A$ for a monic prime polynomial of degree $d$ and denote by $\mathfrak{p}$ the principal ideal that it generates. For $x \in K$, we write $v_{\mathfrak{p}}(x)$ for the valuation of $x$ at $\mathfrak{p}$. 

\begin{definition}
Let $f=\sum_{i=0}^{\infty} c_i u^i$ be a formal series with $c_i \in K$. Then we define the \emph{valuation of $f$ at $\mathfrak{p}$} to be
\begin{equation*}
v_{\mathfrak{p}}(f)=\inf_{i} v_{\mathfrak{p}}(c_i).
\end{equation*}
For two formal series $f=\sum a_iu^i$ and $g=\sum b_i u^i$, we write $f\equiv g \pmod{\mathfrak{p}^m}$ if $v_{\mathfrak{p}}(f-g) \geq m$.
\end{definition}

For any $u$-series $f$ with rational $\mathfrak{p}$-integral coefficients, define its \emph{filtration modulo $\mathfrak{p}$}, denoted $w_{\mathfrak{p}}(f)$, to be the smallest integer $k$ such that there exists a modular form $f'$ of weight $k$ for $\Gl_2(A)$ such that $f \equiv f' \pmod{\mathfrak{p}}$. We write $w_{\mathfrak{p}}(f)=-\infty$ if $f\equiv 0 \pmod{\mathfrak{p}}$. 

As in the classical case, there is a deep connection between supersingular Drinfeld modules in characteristic $\mathfrak{p}$ and forms with lower filtration than weight. It is this connection which we will exploit to refine the connection between the Weierstrass points of $X_0(\mathfrak{p})$ and the supersingular locus.

To begin explaining the connection, let again $g_k$ be the Drinfeld Eisenstein series of weight $q^k-1$ and type $0$ for $\Gl_2(A)$. As shown in \cite{gekeler}, if $\mathfrak{p}$ is an ideal generated by a prime polynomial of degree $d$, we have $g_d \equiv 1 \pmod{\mathfrak{p}}$. Thus the form $g_d$ has filtration equal to $0$, which is strictly less than its weight. We note further that this is the only relation upon reducing modulo $\mathfrak{p}$. 

To connect $g_d$ to the supersingular Drinfeld modules, we must first define the so-called \emph{companion polynomial} to a Drinfeld modular form. In \cite{dobiwagewang}, the authors remark that the fact that the algebra of Drinfeld modular forms for $\Gl_2(A)$ is generated by $g$ and $h$ implies the following: For $k$ a positive integer and $l$ a class in $\mathbb{Z}/(q-1)$, define $\mu(k,l)$ and $\gamma(k,l)$ to be the unique pair of integers such that 
\begin{gather}\label{eq:uniqueness}
\mu(k,l) \equiv l \pmod{q-1},  \notag \\
0\leq \gamma(k,l) \leq q, \\
\text{and} \quad k=\mu(k,l)(q+1)+\gamma(k,l)(q-1). \notag
\end{gather}
Then to every Drinfeld modular form of weight $k$ and type $l$ for $\Gl_2(A)$ one can associate a unique polynomial $P(f,x) \in C[x]$ such that
\begin{equation}\label{jpoly}
f= g^{\gamma(k,l)} h^{\mu(k,l)}P(f,j)
\end{equation}
where $j$ is the (normalized) $j$-invariant, $j=\frac{g^{q+1}}{-h^{q-1}}$. Since $g$ only has a single zero at the elliptic point with $j=0$, the first consequence of this fact is that any form $f$ of a given weight $k$ and type $l$ vanishes to order at least $\gamma(k,l)$ at $j=0$. We will call these zeroes the \emph{trivial zeroes of $f$}. The second consequence of this fact is that since $h$ is nonzero on the Drinfeld upper half-plane, the polynomial $P$ can be thought of as an object which keeps track of the zeroes of the form $f$ that are not trivial.

If we define the \emph{Drinfeld supersingular locus} to be the following polynomial:
\begin{equation*}
S_{\mathfrak{p}}(x) = \prod_{\substack{\phi \text{ defined over } \overline{\mathbb{F}}_{\mathfrak{p}} \\ \phi \text{ supersingular}}} (x-j(\phi)),
\end{equation*}
then we have
\begin{equation*}
S_{\mathfrak{p}}(x) \equiv x^{\gamma(q^d-1,0)} P(g_d,x) \pmod{\mathfrak{p}},
\end{equation*}
where $\gamma(q^d-1,0)$ is $0$ if $d$ is even and $1$ if $d$ is odd. 

Therefore upon reduction modulo $\mathfrak{p}$, the form $g_d$, which has lower filtration than weight modulo $\mathfrak{p}$, has a single zero at each supersingular point. This fact is an example of a more general phenomenon:

\begin{proposition}[Dobi-Wage-Wang\cite{dobiwagewang}]\label{dobiwagewang}
Assuming the notation above, let $f$ be a Drinfeld modular form for $\Gl_2(A)$ of weight $k$ and type $l$ with rational $\mathfrak{p}$-integral $u$-series coefficients and finite filtration $w_{\mathfrak{p}}(f)$. Define $\alpha =\frac{k-w_{\mathfrak{p}}(f)}{q^d-1}$ and $a= \left\lfloor \frac{\alpha \gamma(q^d-1,0)q + \gamma(k,l)}{q+1} \right\rfloor$. Then the polynomial $x^a P(f,x)$ is divisible by $S_{\mathfrak{p}}(x)^{\alpha}$ in $\mathbb{F}_{\mathfrak{p}}[x]$, where $\mathbb{F}_{\mathfrak{p}}$ is the field $A/\mathfrak{p}$.
\end{proposition}






Proposition \ref{dobiwagewang} when applied to a certain Drinfeld modular for $W(z)$ defined in Section \ref{modularwronskian}, immediately implies Theorem \ref{T:main}. To obtain the more precise result given in Theorem \ref{T:ahlgrenono} we will need the following proposition:

\begin{proposition}\label{P:companionpolys}
Let $f$ be a Drinfeld modular form of weight $k$ and type $l$ for $\Gl_2(A)$. 
\begin{enumerate}
\item \label{evens} If $d$ is even, then we have
\begin{equation*}
P(fg_d,x) \equiv P(g_d,x) P(f,x) \pmod{\mathfrak{p}}.
\end{equation*}
\item \label{odds} If $d$ is odd, then 
\begin{equation*}
P(f g_d,x) \equiv 
\begin{cases}
-x P(g_d, x) P(f,x)\pmod{\mathfrak{p}} & \text{if $\gamma(k,l) =q$,}\\
P(g_d,x) P(f,x)\pmod{\mathfrak{p}} & \text{otherwise.}
\end{cases}
\end{equation*}
\end{enumerate}
\end{proposition}

\begin{proof}
\noindent
\underline{The case of $d$ even} Since $g_d \equiv 1 \pmod{\mathfrak{p}}$, we have $f \equiv f g_d \pmod{\mathfrak{p}}$. Furthermore, if $f$ is of weight $k$ and type $l$, then $fg_d$ is of weight $k + q^d-1$ and type $l$. Using the statement of equation (\ref{jpoly}) we have
\begin{equation*}
g^{\gamma(k,l)} h^{\mu(k,l)} P(f,j) \equiv g^{\gamma(k+q^d-1,l)} h^{\mu(k+q^d-1,l)} P(fg_d,j) \pmod{\mathfrak{p}}.
\end{equation*}
Then
\begin{equation*}
P(f g_d, j ) \equiv h^{\mu(k,l)- \mu(k+q^d-1,l)} g^{\gamma(k,l)- \gamma(k+q^d-1,l)} P(f,j) \pmod{\mathfrak{p}}.
\end{equation*}
 
We have $\mu(k,l) \equiv l \equiv \mu(k+q^d-1,l) \pmod{q-1}$, so let $N$ be the integer such that $\mu(k+q^d-1,l) - \mu(k,l) = N(q-1)$. Combining the equations
\begin{equation*}
k = \gamma(k,l)(q-1) + \mu(k,l)(q+1)
\end{equation*}
and
\begin{equation*}
k +q^d-1 = \gamma(k+q^d-1,l)(q-1) + \mu(k+q^d-1,l)(q+1),
\end{equation*}
we obtain that 
\begin{equation}\label{E:unique}
q^d-1 = \left( \gamma(k+q^d-1,l)-\gamma(k,l)\right) (q-1) +N(q-1)(q+1).
\end{equation}

Since both $\gamma(k+q^d-1,l)$ and $\gamma(k,l)$ are between 0 and $q$ inclusively, then
\begin{equation*}
-q \leq \gamma(k+q^d-1,l)-\gamma(k,l) \leq q.
\end{equation*}

If it were the case that
\begin{equation*}
-q \leq \gamma(k+q^d-1,l)-\gamma(k,l) <0,
\end{equation*}
then by the uniqueness of the integers $\mu(q^d-1,0)$ and $\gamma(q^d-1,0)$ in the equation (\ref{E:unique}), we must have 
\begin{equation*}
 \gamma(k+q^d-1,l)-\gamma(k,l) = \gamma(q^d-1,0) - q-1 = -q-1,
\end{equation*}
a contradiction. Therefore 
\begin{equation*}
0 \leq \gamma(k+q^d-1,l)-\gamma(k,l) \leq q,
\end{equation*}
and again using uniqueness in equation (\ref{E:unique}),
\begin{equation*}
0 = \gamma(q^d-1,0) =  \gamma(k+q^d-1,l)-\gamma(k,l),
\end{equation*}
and
\begin{equation*}
N(q-1) = \mu(q^d-1,0).
\end{equation*}

Then 
\begin{equation*}
P(f g_d, j ) \equiv h^{-\mu(q^d-1,0)} P(f,j) \pmod{\mathfrak{p}}.
\end{equation*}

Solving for $P(g_d,j)$ in 
\begin{equation*}
1 \equiv g_d = h^{\mu(q^d-1,0)} P(g_d,j)
\end{equation*}
completes the proof.
\newline

\noindent
\underline{The case of $d$ odd} The proof proceeds as in the even case, except that we cannot rule out the case 
\begin{equation*}
-q \leq \gamma(k+q^d-1,l)-\gamma(k,l) <0.
\end{equation*}
In that case, we must have
\begin{equation*}
\gamma(k+q^d-1,l)-\gamma(k,l) = \gamma(q^d-1,0) - q-1 = 1-q-1 = -q,
\end{equation*}
which forces $\gamma(k,l) = q$. Furthermore, we have
\begin{equation*}
\mu(q^d-1,0) = (N-1)(q-1),
\end{equation*}
where $N$ is such that $\mu(k+q^d-1,l) - \mu(k,l) = N(q-1)$ as in the even case.

Putting this together we have
\begin{align*}
P(f g_d, j ) &\equiv h^{\mu(k,l)- \mu(k+q^d-1,l)} g^{\gamma(k,l)- \gamma(k+q^d-1,l)} P(f,j) \pmod{\mathfrak{p}}\\
 & \equiv h^{-(N-1)(q-1)-(q-1)}g^q P(f,j) \pmod{\mathfrak{p}}
\end{align*}

Multiplying both sides by
\begin{equation}\label{eq:dodd}
1 \equiv g_d = g h^{\mu(q^d-1,0)} P(g_d,j)
\end{equation}
gives
\begin{equation*}
P(f g_d, j) \equiv \frac{g^{q+1}}{h^{q-1}} P(g_d,j) P(f,j) \pmod{\mathfrak{p}},
\end{equation*}
and since $j = \frac{g^{q+1}}{-h^{q-1}}$, this completes the proof of this case.

If 
\begin{equation*}
0 \leq \gamma(k+q^d-1,l)-\gamma(k,l) \leq q,
\end{equation*}
using uniqueness in equation (\ref{E:unique}), we must have
\begin{equation*}
1 = \gamma(q^d-1,0) =  \gamma(k+q^d-1,l)-\gamma(k,l),
\end{equation*}
and
\begin{equation*}
N(q-1) = \mu(q^d-1,0).
\end{equation*}

Then we may conclude similarly as in the even case that
\begin{equation*}
P(f g_d, j ) \equiv g^{-1}h^{-\mu(q^d-1,0)} P(f,j) \pmod{\mathfrak{p}},
\end{equation*}
and the result follows using equation (\ref{eq:dodd}) again.
\end{proof}

We end this subsection by recalling results from \cite{normandtrace} for the convenience of the reader:

\begin{theorem}[Theorem 1.1 of \cite{normandtrace}]\label{tracetheorem}
Let $q\geq 3$. There is a one-to-one correspondence between forms of weight $2$ and type $1$ for $\Gamma_0(\mathfrak{p})$ with rational $\mathfrak{p}$-integral $u$-series coefficients at $\infty$ and forms of weight $q^d+1$ and type $1$ for $\Gl_2(A)$ with rational $\mathfrak{p}$-integral $u$-series coefficients.
\end{theorem}

We will also need a stronger version of Theorem 1.2 from \cite{normandtrace}, and take this opportunity to correct a typo in the type of the form $\widetilde{\N(f)}$:

\begin{theorem}\label{normtheorem}
Let $f$ be a Drinfeld modular form for $\Gamma_0(\mathfrak{p})$ of weight $k$ and type $l$ with rational, $\mathfrak{p}$-integral $u$-series coefficients at $\infty$. Suppose further that $f$ is an eigenform of the Fricke involution. Let
 \begin{equation*}
\widetilde{\N(f)}(z) = \pi^{q^dk/2} \prod_{\gamma \in \Gamma_0(\mathfrak{p})\backslash \Gl_2(A)} f |_{k,l} [\gamma].
 \end{equation*}
Then $\widetilde{\N(f)}$ has rational, $\mathfrak{p}$-integral $u$-series coefficients and
\begin{equation*}
\widetilde{\N(f)} \equiv f^2 \pmod{\mathfrak{p}}.
\end{equation*}
Furthermore, $\widetilde{\N(f)}$ is a form of weight $(q^d+1)k$ and type $2l$.
\end{theorem}

\begin{proof}
We first note that the hypothesis in \cite{normandtrace} that $f$ have integral $u$-series coefficients at $\infty$ is unnecessary; it suffices that the coefficients be rational and $\mathfrak{p}$-integral for all of the arguments in the paper to work.

Corollary 5.4 of \cite{normandtrace} asserts that for $f$ as in the statement of the theorem,
\begin{equation*}
 f(z) \prod_{\substack{\lambda \in A \\ \degree \lambda <d}} f\left( \frac{z+\lambda}{\pi}\right) \equiv f(z)^2 \pmod{\mathfrak{p}}.
\end{equation*}

By Proposition 3.9 of \cite{normandtrace}, 
\begin{equation*}
\N(f) =  \prod_{\gamma \in \Gamma_0(\mathfrak{p})\backslash \Gl_2(A)} f |_{k,l} [\gamma] = \frac{1}{\pi^{q^dk/2}}\; f \prod_{\substack{\lambda \in A \\ \degree \lambda <d}} f\left( \frac{z+\lambda}{\pi}\right),
\end{equation*}
which proves the equivalence modulo $\mathfrak{p}$.

Because $\Gamma_0(\mathfrak{p})$ has index $q^d+1$ in $\Gl_2(A)$, the weight of $\N(f)$ is $(q^d+1)k$, and the type is $(q^d+1)l$. However, the type of a form for $\Gl_2(A)$ is an equivalence class in $\mathbb{Z}/(q-1)$ and
\begin{equation*}
(q^d+1)l  = (q^d-1)l + 2l \equiv 2l \pmod{q-1}.
\end{equation*}
\end{proof}

\subsection{Drinfeld modular curves}\label{xnotp}
We now turn our attention to Drinfeld modular curves, and more specifically to the family $X_0(\mathfrak{p})$.

For $\Gamma$ a congruence subgroup of $\Gl_2(A)$, the action of $\Gamma$ on the Drinfeld upper half-plane $\Omega$ by fractional linear transformations has finite stabilizer for each $z \in \Omega$. It follows thus that the quotient $\Gamma\backslash \Omega$ is also a rigid analytic space. Moreover, it is connected and smooth of dimension one. The curve $\Gamma\backslash \Omega$ can be shown to arise from an algebraic curve:

\begin{theorem}[Drinfeld \cite{drinfeld}]\label{Drinfeld}
There exists a smooth irreducible affine algebraic curve $Y_{\Gamma}$ defined over $C$ such that $\Gamma \backslash \Omega$ and the underlying analytic space $Y_{\Gamma}^{\text{an}}$ of $Y_{\Gamma}$ are canonically isomorphic as analytic spaces over $C$.
\end{theorem}

We note further that the curve $Y_{\Gamma}$ is unique up to isomorphism, and is in fact defined over a finite abelian extension of $K$, $K_{\Gamma}$. For each $Y_{\Gamma}$ there exists a unique smooth projective curve $X_{\Gamma}$ over $K_{\Gamma}$ such that $Y_{\Gamma}$ is birationally equivalent to $X_{\Gamma}$. As sets, $Y_{\Gamma}(C)$ and $X_{\Gamma}(C)$ differ by finitely many points, which are in one-to-one correspondence with the points of the set $\Gamma\backslash \mathbb{P}^1(K)$, where $\gamma = \left(\begin{smallmatrix} a&b\\ c&d \end{smallmatrix}\right) \in \Gamma$ acts on $(x_1:x_2) \in \mathbb{P}^1(K)$ by 
\begin{equation*}
\gamma \cdot (x_1: x_2) = (a x_1 +b x_2: cx_1 + d x_2).
\end{equation*}
These points are called the cusps of $\Gamma$.

For  $\gamma = \left(\begin{smallmatrix} a&b\\ c&d \end{smallmatrix}\right) \in \Gamma$, we have
\begin{equation*}
\frac{d(\gamma z)}{dz}=\det \gamma \,  (cz+d)^{-2},
\end{equation*}
so that for $f$ a modular form for $\Gamma$ of weight $2$ and type $1$, the differential form $f(z)dz$ is $\Gamma$-invariant. A short computation, presented in \cite{gekelerjacobian}, shows that it descends to a holomorphic differential form on $X_{\Gamma}$ if $f$ is a double cusp form. Since GAGA theorems hold for rigid analytic curves \cite{kiehl1} \cite{kiehl2}, we have the following theorem:

\begin{theorem}[Goss \cite{gosseisenstein}, Gekeler-Reversat \cite{gekelerjacobian}]
The map $f \mapsto f(z)dz$ identifies the space of double cusp forms of weight $2$ and type $1$ for $\Gamma$ to the space of regular differential forms on $X_{\Gamma}$.
\end{theorem}

From this theorem it follows that the dimension of the space of double cusp forms of weight $2$ and type $1$ for $\Gamma$ is $g_{\Gamma}$, where $g_{\Gamma}$ is the genus of the curve $X_{\Gamma}$. Furthermore, it follows by a standard argument that all spaces of Drinfeld modular forms of a fixed weight and type for a congruence group $\Gamma$ are finite-dimensional.

We will be interested in one family of congruence subgroups, and the Drinfeld modular curves attached to these groups. Recall that $\pi(T)$ is a monic prime polynomial in $A$ of degree $d$ generating the ideal $\mathfrak{p}$. Then we may define the congruence subgroups
\begin{equation*}
\Gamma=\Gamma_0(\mathfrak{p})\defi \left\{
\begin{pmatrix}
a & b \\
c & d
\end{pmatrix}
\in \Gl_2(A) \mid c \equiv 0 \pmod{\mathfrak{p}} \right\}.
\end{equation*}
In this case, $\# \det \Gamma_0(\mathfrak{p})= q-1$. From now on, we will denote the affine curve $Y_{\Gamma_0(\mathfrak{p})}$ by $Y_0(\mathfrak{p})$ and the projective curve $X_{\Gamma_0(\mathfrak{p})}$ by $X_0(\mathfrak{p})$ to coincide with classical notation. Both $Y_0(\mathfrak{p})$ and $X_0(\mathfrak{p})$ can be defined over $K$, but we will most often think of them as curves over $C$.

As described in \cite{gekelerjacobian}, every congruence subgroup corresponds to a certain moduli problem for Drinfeld modules of rank $2$. The problem attached to $\Gamma_0(\mathfrak{p})$ classifies Drinfeld modules of rank $2$ with a distinguished finite flat subgroup-scheme which is cyclic, locally free of rank $q^d$ and contained in the $\mathfrak{p}$-torsion. We write  $M_0(\mathfrak{p})$ for the coarse moduli scheme associated to this problem.

From Drinfeld's work on ``generalized Drinfeld modules," we may deduce the existence of a compactification $\overline{M}_0(\mathfrak{p})$ of $M_0(\mathfrak{p})$ over $\Spec A$. We have that $X_0(\mathfrak{p})$ as a curve over $K$ is $\overline{M}_0(\mathfrak{p})\times_A K$. From \cite{gekelerhecke} and \cite{drinfeld}, we know that $\overline{M}_0(\mathfrak{p})$ has the following properties:


\begin{theorem}\label{modulischeme}
\begin{itemize}
\item $\overline{M}_0(\mathfrak{p}) \to \Spec A$ is proper, normal, flat, and irreducible, of relative dimension $1$.
\item $\overline{M}_0(\mathfrak{p}) \to \Spec A$ is smooth away from $\mathfrak{p}$.
\item If $d$ is even, $\overline{M}_0(\mathfrak{p})$ is regular. If $d$ is odd, $\overline{M}_0(\mathfrak{p})$ has a singularity on the fiber above $\mathfrak{p}$ at the supersingular $j$-invariant $j=0$, and is otherwise regular. The singularity is of type $A_q$.
\end{itemize}
\end{theorem}


The last part of the theorem requires a careful study of the moduli problem ``in characteristic $\mathfrak{p}$.'' To obtain it, Gekeler \cite{gekelerhecke} shows that the reduction of $X_0(\mathfrak{p})$ modulo $\mathfrak{p}$ is given by two copies of $X_0(1)$ intersecting transversally at the supersingular points and interchanged by the Fricke involution $W_{\mathfrak{p}}$. The Fricke involution can be defined as follows: if $\phi$ is a Drinfeld module and $H$ is a $\Gamma_0(\mathfrak{p})$-level structure, so that $(\phi, H)$ is a point of $M_0(\mathfrak{p})$, then $W_{\mathfrak{p}}(\phi, H) = (\phi / H, \phi[\mathfrak{p}]/ H)$. 

\begin{remark}\label{xnotintegrality}
From Theorem \ref{modulischeme} above, we have that $X_0(\mathfrak{p})$ is defined over $K$ with function field $K(j, j_{\mathfrak{p}})$. In fact, because the moduli problem associated to $\Gamma_0(\mathfrak{p})$ is defined over $A$, the space of holomorphic differentials on $X_0(\mathfrak{p})$ has a basis that is defined over $A$. Therefore, the space of Drinfeld double cusp forms of weight $2$ and type $1$ for $\Gamma_0(\mathfrak{p})$ has a basis of forms with integral coefficients. It also follows from such considerations that Drinfeld modular forms on $\Gamma_0(\mathfrak{p})$ with rational $u$-series coefficients have bounded denominators.
\end{remark}

From its action on pairs $(\phi, H)$, we can also see that the Fricke involution $W_{\mathfrak{p}}$ is $K$-rational. We note here that the analytic avatar of $W_{\mathfrak{p}}$ is the action of the matrix $\left(\begin{smallmatrix} 0&-1\\ \pi&0 \end{smallmatrix}\right)$ on $\Omega$.

Since $X_0(\mathfrak{p})$ is smooth, its arithmetic and geometric genera are the same and do not depend on the field over which we consider the curve. We denote the genus of $X_0(\mathfrak{p})$ by $g_{\mathfrak{p}}$, and it is given by
\begin{equation*}
g_{\mathfrak{p}} =
\begin{cases}
\frac{q(q^{d-1}-1)}{q^2-1} & \qquad \mbox{if $d$ is odd,}\\
\frac{q^2(q^{d-2}-1)}{q^2-1} & \qquad \mbox{if $d$ is even.}
\end{cases}
\end{equation*}
(As promised, this is the same $g_{\mathfrak{p}}$ that appears in equation (\ref{genusformula}).) This fact can be obtained either by relating $g_{\mathfrak{p}}$ to $h_1(\Gamma_0(\mathfrak{p})) \backslash \mathcal{T})$ as in \cite{gekeler17}, or by working directly on the Drinfeld modular curve as in \cite{gekeler11}.

From \cite{gekeler11}, we also note that representatives for the two distinct equivalence classes of $\Gamma_0(\mathfrak{p})\backslash \mathbb{P}^1(K)$ are $(0:1)$ and $(1:0)$, so that $X_0(\mathfrak{p})$ has two cusps, denoted $0$ and $\infty$, respectively. Both of these cusps are $K$-rational points of $X_0(\mathfrak{p})$. From the same source, we have that $X_0(\mathfrak{p})(C)$ has no elliptic point when $d$ is odd, and two elliptic points when $d$ is even. When $d$ is even, both elliptic points have stabilizer of order $q+1$ in $\widetilde{\Gamma_0}(\mathfrak{p})=\Gamma_0(\mathfrak{p})/\left(\Gamma_0(\mathfrak{p}) \cap Z(\Gl_2(A))\right)$.

\subsection{Expansions at the cusps}\label{infinity}

Some care is needed in discussing the behavior of Drinfeld modular forms at the cusps, so we delve into this topic now. We focus on the groups $\Gl_2(A)$ and $\Gamma_0(\mathfrak{p})$ as this is all we will need here, and leave the general case to \cite{gekelerjacobian} or \cite{gekelermodularcurve}. 

Let us first consider the case of $\Gl_2(A)$. The set $\Gl_2(A) \backslash \mathbb{P}^1(K)$ consists of a single element, and we choose $(1:0)$ as the representative of this element. The stabilizer $\Gamma_{\infty}$ of $(1:0)$ in $\Gl_2(A)$ is the set of all upper-triangular matrices. This set contains a maximal subgroup $\Gamma_{\infty}^{un}$:
\begin{equation*}
\Gamma_{\infty}^{un} = \left\{
\begin{pmatrix}
1& a \\
0 & 1
\end{pmatrix}
: a \in A \right\},
\end{equation*}
and also cyclic transformations $\left(\begin{smallmatrix} a & 0 \\ 0 & d \end{smallmatrix}\right)$ for $a, d \in \mathbb{F}_q^{\times}$. The image of this group of cyclic transformations in $\Pgl_2(A)$ has size $q-1$, the size of $\mathbb{F}_q^{\times}$.

Recall the function $u$ defined in equation (\ref{u}). Now writing
\begin{equation*}
\Omega_c= \{z\in \Omega : \operatorname{inf}_{x \in K_{\infty}} |z-x| \geq c \},
\end{equation*} 
we have that $u$ identifies $\Gamma_{\infty}^{un}\backslash \Omega_c$ with a pointed ball $B_r - \{0\}$ of radius $r$ for some small $r$ \cite{gekelerjacobian}. It can be shown that there is a constant $c_0$ such that for $c \geq c_0$ and $\gamma \in \Gl_2(A)$, $\Omega_c \cap \gamma(\Omega_c) \neq \emptyset$ implies that $\gamma \in \Gamma_{\infty}$. Thus for such a $c$,
\begin{alignat*}{2}
B_{r^{q-1}} - \{0\} &\cong \Gamma_{\infty} \backslash \Omega_c && \hookrightarrow \Gl_2(A) \backslash \Omega \\
u(z)^{q-1} &\gets z & &\to z
\end{alignat*}
is an open immersion of analytic spaces. Thus $u(z)^{q-1}$ is a uniformizer at the cusp $\infty$ for $\Gl_2(A) \backslash \Omega$.

The subtlety involved in defining the $u$-series expansion of a Drinfeld modular form is that we allow them to have non-trivial type $l$, and thus they are not invariant under the full $\Gamma_{\infty}$, but rather only under $\Gamma_{\infty}^{un}$. This is why in general a Drinfeld modular form of non-trivial type will have a $u$-series expansion rather than a $u^{q-1}$-series expansion.

There is also a second subtlety that comes into play. For a general congruence subgroup $\Gamma$, to discuss the behavior of a function $f$ at a cusp $(a:b) \in \Gamma \backslash \mathbb{P}^1(K)$, one first fixes an element $\gamma \in \Gl_2(K)$ such that $\gamma \cdot (1:0) = (a:b)$. Then the holomorphy properties and order of vanishing of $f$ at the cusp corresponding to $(a:b)$ are the properties of $f \circ \gamma$ at $\infty$, and do not depend on the choice of $(a:b)$ in its equivalence class modulo $\Gamma$ and on the choice of $\gamma$ sending $(1:0)$ to $(a:b)$. However, for $t$ a parameter at $\infty$ for the group $\Gamma$, one might wish to define the $t$-series expansion of $f$ at the cusp corresponding to $(a:b)$ as that of $f\circ \gamma$ at $\infty$. This is not well-defined, as the coefficients of the expansion will depend on the choice of $(a:b)$ and $\gamma$. 

To remove any ambiguity, in the case of $\Gl_2(A)$ we once and for all declare that the expansion of $f$ at $\infty$ is its $u$-series expansion, with $u$ as defined in equation (\ref{u}).

We now consider $\Gamma_0(\mathfrak{p})$. The cusp in the $\Gamma_0(\mathfrak{p})$-equivalence class of $(1:0)$, which we denote by $\infty$, has stabilizer $\Gamma_{\infty}$ in $\Gamma_0(\mathfrak{p})$, where $\Gamma_{\infty}$ is again the set of all upper-triangular matrices in $\Gl_2(A)$. Because of this, the same argument as above shows that $u^{q-1}$ is a parameter at $\infty$, and that modular forms for $\Gamma_0(\mathfrak{p})$ have a $u$-series expansion at $\infty$. As in the case of $\Gl_2(A)$, we fix once and for all that the expansion of $f$ at $\infty$ is its $u$-series expansion.

We now consider the other cusp of $X_0(\mathfrak{p})$, which we will denote by $0$. To fix a well-defined choice of $u$-series expansion at $0$, we fix $(0:1)$ as the representative of the other equivalence class, and the matrix
\begin{equation*}
W_{\mathfrak{p}}=
\begin{pmatrix}
0 & -1 \\
\pi & 0
\end{pmatrix}
\end{equation*}
as the matrix sending $(1:0)$ to $(0:1)$. Thus the $u$-series expansion of a Drinfeld modular form of weight $k$ and type $l$ at the cusp $0$ is defined to be that of the form
\begin{equation*}
f|_{k,l} [W_{\mathfrak{p}}] = \pi^{k/2}(\pi z)^{-k}f\left(\frac{-1}{\pi z}\right)
\end{equation*}
at $\infty$.

In any case, for a Drinfeld modular form with $u$-series expansion $\sum_{i=0}^{\infty}a_i u(z)^i$ at a cusp $c$, we will write $\order_c(f)$ for the least $i\geq 0$ such that $a_i \neq 0$, and call this the order of vanishing of $f$ at $c$.

\section{Hyperderivatives and quasimodular forms}\label{quasimodular}

In this section we present the theory necessary to study the action of differential operators on the algebra of Drinfeld modular forms. These operators will not preserve modularity, which naturally leads us to consider a larger set of rigid analytic functions on $\Omega$, the Drinfeld quasimodular forms. Throughout, we will use ``analytic'' to mean ``rigid analytic." We will say that an analytic function $f$ on $\Omega$ is ``analytic at $\infty$'' to mean that there are constants $a_i \in C$, $i \in \mathbb{Z}_{\geq 0}$ such that 
\begin{equation*}
f(z)=\sum_{i=0}^{\infty}a_i u(z)^i
\end{equation*}
for $z$ such that $\operatorname{inf}_{x \in K_{\infty}} |z-x| $ is large. 

\subsection{Drinfeld quasimodular forms}

\begin{definition}
An analytic function $f \colon \Omega \rightarrow C$ is called a \emph{Drinfeld quasimodular form of weight $k$, type $l$, and depth $m$ for $\Gl_2(A)$}, where $k \geq 0$ and $m \geq 0$  are integers and $l$ is a class in $\mathbb{Z}/(q-1)$, if there exist analytic functions $f_1$, $f_2, \ldots, f_m$ on $\Omega$ which are $A$-periodic and analytic at infinity such that for $\gamma = \left(\begin{smallmatrix} a&b\\ c&d \end{smallmatrix}\right) \in \Gl_2(A)$, we have
\begin{equation*}
f(\gamma z)=(\det \gamma)^{-l}(cz+d)^k \sum_{j=0}^{m} f_j(z) \left( \frac{c}{cz+d}\right)^j.
\end{equation*}
\end{definition}

For a given quasimodular form $f \neq 0$, the weight, type and polynomial $\sum_{j=0}^{m} f_j(z)X^j$ are uniquely determined by $f$ as shown in \cite{bosserp}. Furthermore, as can be seen by choosing $\gamma$ to be the identity matrix, we necessarily have $f=f_0$. Finally, every modular form is a quasimodular form of depth $0$, and vice-versa.

An important example of a Drinfeld quasimodular form is the function $E$ introduced in \cite{gekeler}:
\begin{equation*}
E \defi \frac{1}{\tilde{\pi}}\sum_{\substack{a \in \mathbb{F}_q[T]\\ a \text{ monic}}}\left(\sum_{b \in \mathbb{F}_q[T]} \frac{a}{az+b} \right),
\end{equation*}
which can be shown to be of weight $2$, type $1$ and depth $1$. Its importance is reflected in the fact that the graded $C$-algebra of Drinfeld quasimodular forms of all weights, types and depths is the polynomial ring $C[g,h,E]$, where each form corresponds to a unique isobaric polynomial. 

For a more in-depth discussion of Drinfeld quasimodular forms, we refer the interested reader to the work of Bosser and Pellarin \cite{bosserp} and \cite{bosserp2}.

\subsection{Higher derivatives}\label{analytic}
In \cite{uchinosatoh}, Uchino and Satoh consider the action of the Hasse derivatives on analytic functions on $\Omega$. We present here the results we need from their paper without proof.

We will use the fact that $C$ is a complete field with a non-Archimedean dense valuation (which we recall is the unique extension of $v_{\infty}(x)=-\operatorname{deg}(x)$ from $K$ to $C$) and that $\Omega$ is an open set. We will work in this section with analytic functions on $\Omega$ and denote the space of these functions by $\An(\Omega)$. For $f \in \An(\Omega)$ such that $f=\sum_{i=0}^{\infty} c_{i,w}(z-w)^i$ in a neighborhood of $w\in \Omega$, we define the $n^{th}$ hyperderivative of $f$ at $w$ to be
\begin{equation}
\mathfrak{D}_n(f)(w)=c_{n,w}.
\end{equation}
As remarked above, this is simply the Hasse derivative.

For our purposes, it will be important that our differential operator preserves $K$-rationality of the $u$-series coefficients, which $\mathfrak{D}_n$ does not. However, the operator
\begin{equation}
D_n \defi \frac{1}{(-\tilde{\pi})^n}\mathfrak{D}_n
\end{equation}
does \cite{bosserp}, and so we will use this normalized operator.

\begin{remark}
The operator $-D_1$ was also studied by Gekeler in \cite{gekeler}, where it was denoted by $\Theta$, in analogy with Ramanujan's $\Theta$-operator in the classical setting. This explains the discrepancy in sign between this work and the cited paper in our statement of Proposition \ref{partial} below.
\end{remark}

We have the following facts:

\begin{proposition}[Uchino-Satoh \cite{uchinosatoh}]\label{uchinosatoh}
For $f\in \An(\Omega)$ and $w \in \Omega$ such that $f=\sum_{i=0}^{\infty} c_{i,w}(z-w)^i$ near $w$, we have:
\begin{enumerate}
\item \label{Dnexpansion} Formally, in a neighborhood of $w$,
\begin{equation*}
D_nf (z)=\frac{1}{(-\tilde{\pi})^n}\sum_{i=0}^{\infty} \binom{i}{n}c_{i,w}(z-w)^{i-n}
\end{equation*}
and this has the same radius of convergence as $\sum_{i=0}^{\infty} c_{i,w}(z-w)^i$.
\item In fact, $D_n f$ is analytic on $\Omega$.
\item The system of derivatives $\{D_n\}$ is a \emph{higher derivation}; in other words it satisfies:
\begin{enumerate}
\item $D_0 f=f$,
\item $D_n$ is $C$-linear,
\item for $f$ and $g$ in $\An(\Omega)$, $D_n(fg)=\sum_{i=0}^{n}D_i f D_{n-i} g$.
\end{enumerate}
\item This higher derivation is \emph{iterative}: for all integers $i\geq 0$ and $j\geq 0$, we have:
\begin{equation*}
D_i \circ D_j = D_j \circ D_i = \binom{i+j}{i}D_{i+j}.
\end{equation*}
\item This higher derivation has a \emph{chain rule property}: For each $n\geq 1$ and each $1\leq i \leq n$, there exist maps $F_{n,i}$ from $\An(\Omega)^{n+1-i}$ to $\An(\Omega)$ such that:
\begin{enumerate}
\item for $f$ and $g$ in $\An(\Omega)$ such that the composition $f \circ g$ is defined, we have
\begin{equation*}
D_n(f \circ g)=\sum_{i=1}^{n}F_{n,i}(D_1g, \ldots, D_{n+1-i} g)(D_i f) \circ g,
\end{equation*}
\item and if $n\geq 2$, then $F_{n,1}$ is a $C$-linear map.
\end{enumerate}
\end{enumerate}
\end{proposition}

In the case where $g$ is a linear fractional transformation, \cite[Lemma 3.3]{bosserp} gives the following more precise formula for the maps $F_{n,i}$ that appear in the chain rule property:
\begin{lemma}\label{prop:quasimodular}
Let $f \colon \Omega \to C$ be an analytic function. For all $n \geq 1$, $z \in \Omega$, and $\gamma = \left(\begin{smallmatrix} a&b\\ c&d \end{smallmatrix}\right) \in \Gl_2(A)$, we have
\begin{equation*}
D_n(f \circ \gamma)(z)=(-1)^n \left( \frac{c}{cz+d}\right)^n \sum_{i=1}^{n} (-1)^i \binom{n-1}{n-i} \left( \frac{c(cz+d)}{\det \gamma}\right)^{-i} \frac{1}{(-\tilde{\pi})^{n-i}}(D_i f)\left(\frac{az+b}{cz+d}\right).
\end{equation*}
\end{lemma}

We note here that since the $D_n$'s are iterative and using Lucas' theorem, we have that
\begin{equation}
D_n = \frac{1}{n_0! \ldots n_s!}D^{n_s}_{p^s} \circ \ldots \circ D^{n_1}_p \circ D^{n_0}_1,
\end{equation}
for $n=n_s p^s+ \cdots +n_1 p+ n_0$ the representation of $n$ in base $p$, with $0 \leq n_j \leq p-1$ for each $j$, and where the exponent of $n_j$ on $D_{p^j}$ denotes the $n_j$-fold composition.

As remarked at the beginning of this section, the $D_n$'s do not preserve modularity, but they do preserve quasimodularity, as shown in \cite{bosserp}. For our purposes we shall only need this weaker version of their more general theorem:

\begin{proposition}
Let $f$ be a modular form of weight $k$ and type $l$ for $\Gl_2(A)$. Then for all $n\geq 0$ $D_n f$ is $A$-periodic and analytic at $\infty$, and  for $\gamma = \left(\begin{smallmatrix} a&b\\ c&d \end{smallmatrix}\right) \in \Gl_2(A)$, we have
\begin{equation}
D_n f(\gamma z)= (cz+d)^{k+2n} (\det \gamma)^{-l-n} \sum_{j=0}^{n}\binom{n+k-1}{j}\frac{D_{n-j} f(z)}{(-\tilde{\pi})^j} \left(\frac{c}{cz+d}\right)^j.
\end{equation}
In other words, the function $D_n f$ is a quasi-modular form of weight $k+2n$, type $l+n$ and depth $n$.
\end{proposition}

\subsection{Integrality and vanishing results}

For $i \in \mathbb{N}$, write $[i]=T^{q^i}-T$, the product of all monic prime polynomials of degree dividing $i$, $d_i=[1]^{q^{i-1}} \cdots [i-1]^q[i]$, the product of all monics of degree $i$, and $d_0=1$. In \cite{bosserp}, Bosser and Pellarin obtain the following result on the action of the $D_n$'s on the $u$-series coefficients of quasimodular forms:

\begin{proposition}\label{formuladn}
Let $f \in \An(\Omega)$ be analytic at $\infty$ with $u$-series expansion $f(z)=\sum_{i\geq 0} a_i u^i$. Then for all $n \geq 0$ we have $D_n f(z)= \sum_{i\geq 2} b_{n,i} u^i$, where
\begin{equation}\label{bs}
b_{n,i}=\sum_{r=1}^{i-1}(-1)^{n+r}\binom{i-1}{r}\left(\sum_{\substack{n_1, \ldots, n_r \geq 0 \\ q^{n_1}+ \cdots +q^{n_r}=n}} \frac{1}{d_{n_1} \cdots d_{n_r}}\right) a_{i-r}.
\end{equation}
\end{proposition}

From this explicit formula we can clearly see that
\begin{corollary}\label{derivativesok2}
For $n<q^e$, the operator $D_n$ preserves $\mathfrak{p}$-integrality of the $u$-series coefficients for all prime ideals $\mathfrak{p}$ generated by a prime polynomial of degree $\geq e$.
\end{corollary}

\begin{proof}
Let $e$ be a positive integer. If $n<q^e$, then we have $n_j<e$ for each $n_j$ appearing in the sum defining the $b_{n,i}$'s in equation (\ref{bs}). Since $d_{n_j}$ is only divisible by primes of degree $\leq n_j$, for $n<q^e$ $D_n$ introduces only denominators of degree $<e$.
\end{proof}

From this it easily follows that

\begin{corollary}\label{derivativesok}
Suppose that $f \equiv f' \pmod{\mathfrak{p}}$ for $\mathfrak{p}$ generated by a prime of degree $d$. Then $D_n(f) \equiv D_n(f') \pmod{\mathfrak{p}}$ for $n < q^d$.
\end{corollary}



We will also need:

\begin{proposition}\label{orderomega}
Let $w \in \Omega$ and $f \in \An(\Omega)$, then $\order_{w} D_n(f) \geq \order_{w}(f)-n$. When $n\leq \order_{w}(f)$, we have equality if and only if $\binom{\order_{w}(f)}{n}\not \equiv 0 \pmod{p}$.
\end{proposition}

\begin{proof}
This follows by Proposition \ref{uchinosatoh} part \ref{Dnexpansion}.
\end{proof}

\subsection{A computational tool}\label{tool}

The action of $D_n$ quickly becomes difficult to compute explicitly as $n$ grows. A better-behaved operator was defined by Serre in the classical case (see \cite{kanekokoike}), and we will use its analogue in the Drinfeld setting. Let $n$ and $d$ be non-negative integers. The \emph{$n^{\text{th}}$ Serre's operator of degree $d$} is defined by the formula:
\begin{equation}
\partial_n^{(d)}f=D_nf+\sum_{i=1}^{n}(-1)^i\binom{d+n-1}{i}(D_{n-i}f)(D_{i-1}E).
\end{equation}
In~\cite{bosserp2}, the authors show that $\partial_n^{(k)}$ sends Drinfeld modular forms of weight $k$ and type $l$ to Drinfeld modular forms of weight $k+2n$ and type $l+n$. 



For simplicity we will denote the operator $\partial_1^{(k)}$ by $\partial$, and make the convention that if $f$ is a Drinfeld modular form of weight $k$, then $\partial (f)=\partial_1^{(k)}(f)$. Then for $f$ a Drinfeld modular form of weight $k$,
\begin{equation*}
\partial (f)= D_1(f) - k E f.
\end{equation*}

We have the following:
\begin{proposition}[Gekeler \cite{gekeler}]\label{partial}
\textbf{}
\begin{enumerate}
\item Let $f_i$ for $i =1, 2$ be Drinfeld modular forms of weight $k_i$, then $\partial(f_1f_2)=\partial(f_1)f_2+f_1\partial(f_2)$.
\item $\partial(g)=-h$ and $\partial(h)=0$.
\end{enumerate}
\end{proposition}

This proposition allows us to compute the action of $\partial$ on all Drinfeld modular forms, since $g$ and $h$ generate the algebra of Drinfeld modular forms. Furthermore, since $D_n(E)=E^{n+1}$ for $1\leq n <p$, a tedious but easy computation shows that for a Drinfeld modular form $f$ of weight $k$, we have
\begin{equation}
\partial^n f = n! \partial_n^{(k)} f
\end{equation}
for $1\leq n < p$, where again the exponent on $\partial$ on the lefthand side denotes $n$-fold composition of the $\partial$ operator. This relation in fact holds for $p \leq n < q$ as well, which simply implies that the $n$-fold composition of $\partial$ beyond $\partial^{p-1}$ is identically zero, as expected in characteristic $p$.

\section{Weierstrass points on $X_0(\mathfrak{p})$}\label{weierstrass}

\subsection{Previous results}

As discussed in Section \ref{theory}, crucial to the study of Weierstrass points in positive characteristic is the knowledge of the curve's canonical gap sequence. 



\begin{proposition}[Armana, personal communication]\label{armana}
Let $\mathfrak{p}$ be a prime ideal generated by a polynomial of degree at least $3$ in $\mathbb{F}_q[T]$. Then $X_0(\mathfrak{p})$ has a classical gap sequence.
\end{proposition} 

\begin{proof}
Recall from Section \ref{theory} that if $X$ is a smooth projective irreducible curve defined over an algebraically closed field that has a classical gap sequence, then the osculation points and the Weierstrass points of $X$ coincide; if $X$ does not have a classical gap sequence then every point of $X$ is an osculation point.

Using an argument analogous to Ogg's argument in the classical case, Armana \cite{armanaphd} shows the following: Let $P$ be a $K$-rational point of $X_0(\mathfrak{p})$ such that its unique extension to a section of $\overline{M}_0(\mathfrak{p})$ over $A$ is not supersingular at $\mathfrak{p}$, and denote by $c\geq 1$ the smallest pole number at $P$. Then $c \geq 1+ g_{\mathfrak{p}}$, where as before $g_{\mathfrak{p}}$ is the genus of $X_0(\mathfrak{p})$.

We repeat her argument here since \cite{armanaphd} is in French: Let $P$ be such a point, and let $c \geq 1$ be an integer such that $c$ is a pole number of $P$; recall that this means that there is a function $F$ on $X$ that has a pole of order $c$ at $P$ and is regular elsewhere. Since $P$ is $K$-rational, we may suppose that $F$ is defined over $K$ as well. The Fricke involution $W_{\mathfrak{p}}$ of $X_0(\mathfrak{p})$ is also defined over $K$, and we write $P' = W_{\mathfrak{p}}(P)$; $P'$ is also $K$-rational. Up to adding to $F$ a constant belonging to $K$, we may suppose that $f(P') =0$. 

As stated in Section \ref{xnotp} the reduction of $X_0(\mathfrak{p})$ modulo $\mathfrak{p}$ is given by two copies $Z$ and $Z'$ of $X_0(1)$ intersecting transversally at the $g_{\mathfrak{p}}+1$ supersingular points over the algebraic closure of $A/\mathfrak{p}$ and interchanged by the Fricke involution $W_{\mathfrak{p}}$. Without loss of generality, suppose that the reduction modulo $\mathfrak{p}$ of $P$, which we denote $\tilde{P}$, belongs to $Z$ and the reduction modulo $\mathfrak{p}$ of $P'$, denoted $\tilde{P'}$, belongs to $Z'$. Up to multiplication by a constant in $K^{\times}$, we may suppose that the reduction modulo $\mathfrak{p}$ of $F$, $\tilde{F}$, is reduced and non-constant.

On $Z'$, $\tilde{F}$ has a zero at $\tilde{P'}$ and no pole since $\tilde{P}$ is not supersingular and therefore does not belong to $Z'$. Therefore $\tilde{F}$ is identically zero on $Z'$. In particular, $\tilde{F}$ vanishes at each supersingular point. On $Z$, the restriction of $\tilde{F}$ has at least $g_{\mathfrak{p}}+1$ zeroes and at most a pole at $P$ of order $c$. Since the degree of the divisor of a function is zero, it follows that $g_{\mathfrak{p}}+1 \leq c$.

It suffices now to notice that such a point is not an osculation point of the curve. Either one of the cusps of $X_0(\mathfrak{p})$ satisfies the conditions on the point $P$ above. Thus $X_0(\mathfrak{p})$ has a point that is not an osculation point, and the result follows.
\end{proof}

\begin{remark}
The requirement that $\mathfrak{p}$ be generated by a prime polynomial of degree at least $3$ ensures that $X_0(\mathfrak{p})$ has genus at least 2. (See equation \ref{genusformula} for an expression giving the genus of $X_0(\mathfrak{p})$ as it depends on the degree $d$ of the prime polynomial generating $\mathfrak{p}$ and Remark \ref{genusconstraint} for an explanation of the requirement that the genus be at least 2.)
\end{remark}

It is immediate from the proof of Proposition \ref{armana} above that the $K$-rational Weierstrass points of $X_0(\mathfrak{p})$ have supersingular reduction modulo $\mathfrak{p}$. A stronger result can be deduced using the following theorem:

\begin{theorem}[Baker \cite{baker}]\label{baker}
Let $R$ be a complete discrete valuation ring with algebraically closed residue field $k$. Let $X$ be a smooth, proper, geometrically connected curve defined over the fraction field of $R$, and denote by $\mathfrak{X}$ a proper model for $X$ over $R$. (In other words, $\mathfrak{X}$ is a proper flat scheme over $\operatorname{Spec} R$ such that its generic fiber is $X$.) Suppose that the special fiber of $\mathfrak{X}$ consists of two genus $0$ curves intersecting transversally at $3$ or more points. Then every Weierstrass point of $X$ defined over the fraction field of $R$ specializes to a singular point of the special fiber of $\mathfrak{X}$.
\end{theorem}

The proof of this Theorem is a corollary of a Specialization Lemma proved in the same paper, which roughly says that the dimension of a linear system can only increase under specialization from the curve $X$ to the dual graph of the model $\mathfrak{X}$. 

Let $K_{\mathfrak{p}}^{un}$ denote the maximal unramified extension of the field $K_{\mathfrak{p}}$, where $K_{\mathfrak{p}}$ the completion of $K$ at $\mathfrak{p}$. Baker's theorem implies that the $K_{\mathfrak{p}}^{un}$-rational Weierstrass points of $X_0(\mathfrak{p})$ have supersingular reduction modulo $\mathfrak{p}$ since $\overline{M}_0(\mathfrak{p})$ satisfies the hypothesis of the theorem when considered as a scheme over the ring of integers of $K_{\mathfrak{p}}^{un}$.


\subsection{The modular Wronskian}\label{modularwronskian}

It is natural to ask whether it is possible to say more about the connection between the supersingular locus at $\mathfrak{p}$ and the Weierstrass points of $X_0(\mathfrak{p})$, as was done in the classical case by Rohrlich \cite{rohrlich2}, and Ahlgren and Ono \cite{ahlgrenono}. To refine the connection, we now develop the ideas of Section \ref{theory} for the curve $X_0(\mathfrak{p})$ over $C$ using Drinfeld modular forms, as Rohrlich did in the classical setting. 

We consider the rigid analytic structure on $X_0(\mathfrak{p})$, so that we can compute with Drinfeld modular forms. For ease of reading, we will continue to write analytic below to mean rigid analytic. An analytic function without poles will be a holomorphic function, and an analytic function possibly with poles will be said to be meromorphic.

We first note that GAGA theorems hold for rigid analytic geometry \cite{kiehl1}, \cite{kiehl2}. More precisely, we will need the following: Let $X$ be a smooth projective algebraic curve defined over a complete non-archimedean field $k$ of positive characteristic $p$, and let $X^{an}$ be the rigid analytic space associated to $X$. (See for example \cite{fresnel} for the construction of $X^{an}$). Then there is an equivalence of category between the algebraic coherent sheaves on $X$ and the analytic coherent sheaves on $X^{an}$. Using this correspondence we will associate to an algebraic coherent sheaf $F$ on $X$ an analytic coherent sheaf denoted $F^{an}$ on $X^{an}$.

We note that the sets of points (throughout we will consider only $C$-valued points) of $X$ and $X^{an}$ coincide, so that we will not make a distinction between a divisor on $X$ and a divisor on $X^{an}$. We denote by $O$ the sheaf of algebraic regular functions on $X$, and by $\mathcal{O}$ the sheaf of holomorphic functions on $X^{an}$.

The linear space $L(D)$ associated to a divisor $D$ on $X$ is the space of global sections of an algebraic sheaf which we will also denote by $L(D)$. The sheaf $L(D)$ is coherent and thus corresponds to a sheaf $L(D)^{an}$ on $X^{an}$.

Because the operation $*^{an}$ commutes with duals and tensor products, $L(D)^{an}$ is none other than $\mathcal{L}(D)$, the subsheaf of meromorphic functions $\mathcal{M}$ on $X^{an}$ such that for $U$ an open set of $X^{an}$ we have
\begin{equation*}
\mathcal{L}(D)(U)= \{ f \in \mathcal{M}(U) \mid  [f] \geq -D|_{U} \} \cup \{0\}.
\end{equation*}

In particular, by GAGA, the space of global sections of $L(D)$ is isomorphic to the space of global sections of $\mathcal{L}(D)$, and for a point $P$ of $X^{an}$ we may instead consider the sequence of spaces
\begin{equation*}
k= \mathcal{L}(0)(X^{an}) \subseteq \mathcal{L}(P)(X^{an}) \subseteq \mathcal{L}(2P)(X^{an}) \subseteq \mathcal{L}(3P)(X^{an}) \subseteq \ldots
\end{equation*}
Then $L((n-1)P)(X)=L(nP)(X)$ if and only if $\mathcal{L}((n-1)P)(X^{an})=\mathcal{L}(nP)(X^{an})$, so that the gap sequences can be computed analytically. 

Denote by $C_{can}$ a canonical divisor on $X$. Arguing as in the algebraic case, if $j$ is a canonical order at $P$, there is $F \in \mathcal{L}(C_{can})(X^{an})$ such that $v_P(F)=j-v_P(C_{can})$. 

We now start our work on $X_0(\mathfrak{p})$ in earnest. Our task now is to define a Drinfeld modular form $W(z)$ for $\Gamma_0(\mathfrak{p})$ whose divisor will capture information about the Weierstrass points of $X_0(\mathfrak{p})$. We note that since the cusps of $X_0(\mathfrak{p})$ are not Weierstrass points, to obtain our main result it is enough to consider the divisor of $W(z)$ away from the cusps. In Section \ref{divisorcusp}, we will collect what we know of the divisor of $W(z)$ at the cusp $\infty$. We recall that $Y_0(\mathfrak{p})$ denotes the affine curve whose $C$-points are exactly those of $X_0(\mathfrak{p})$, but with the cusps excluded.


We first compute the divisor of $dz$ away from the cusps, where $z$ is a parameter on $\Omega$: Let $P \in Y_0(\mathfrak{p})$, and choose $\tau \in \Omega$ to be a representative of $P$ in the Drinfeld upper half-plane. Throughout we write $e_{\tau}$ for the order of the stabilizer of $\tau$ in 
\begin{equation*}
\widetilde{\Gamma_0}(\mathfrak{p})=\Gamma_0(\mathfrak{p})/\Gamma_0(\mathfrak{p}) \cap Z(\Gl_2(A)).
\end{equation*}
Then we may choose $t=(z-\tau)^{e_{\tau}}$ as an analytic parameter at $P$. We have 
\begin{equation*}
dz=\frac{1}{e_{\tau}}t^{(e_{\tau}-1)/e_{\tau}}dt
\end{equation*}
($e_{\tau}$ is either 1 or $q+1$ \cite{gekelerjacobian} so it is prime to the characteristic $p$ of $C$) and so $dz$ has a pole of order 
\begin{equation*}
\frac{e_{\tau}-1}{e_{\tau}}
\end{equation*}
at $\tau$.

\begin{proposition}\label{goodbasis}
Let $P$ be a point on $Y_0(\mathfrak{p})$, and write $j_0(P)=0$, and $(j_1(P), \ldots, j_{g_{\mathfrak{p}}-1}(P))$ for the canonical orders at $P$. Choose $\tau \in \Omega$ to be a representative of $P$ in the Drinfeld upper half-plane, and write $e_{\tau}$ for the order of the stabilizer of $\tau$ in $\widetilde{\Gamma_0}(\mathfrak{p})$. Then there is a basis $\{f_i\}_{i=0}^{g_{\mathfrak{p}}-1}$ of $M^2_{2,1}(\Gamma_0(\mathfrak{p}))$ such that:
\begin{equation*}
\order_{\tau}(f_i)= e_{\tau}j_i(P)+e_{\tau}-1.
\end{equation*}
for each $i$.
\end{proposition}

\begin{proof}
Fix a point $P$ on $Y_0(\mathfrak{p})$, and let $s$ be a parameter at $P$. We choose as our canonical divisor the divisor $[ds]$. There is a basis $\{F_0, \ldots, F_{g_{\mathfrak{p}}-1}\}$ of $\mathcal{L}([ds])$ such that $\order_P(F_i)=j_i(P)$. Furthermore, $\{F_i ds \}$ is a basis for the space of analytic regular differentials $H^0(X_0(\mathfrak{p})^{an}, \Omega^1_{an})$ on $X_0(\mathfrak{p})$. Because of the correspondence between the space $M^2_{2,1}(\Gamma_0(\mathfrak{p}))$ of double cusp forms of weight $2$ and type $1$ for $\Gamma_0(\mathfrak{p})$ and the space of analytic regular differentials on $X_0(\mathfrak{p})$, we have that there is a basis $\{f_i\}$ for $M^2_{2,1}(\Gamma_0(\mathfrak{p}))$ such that $f_i(z)dz=F_ids$. In particular, $\order_P(f_i(z)dz)=\order_P(F_ids)=\order_P(F_i)=j_i(P)$. 

We now use the fact that for $P\in Y_0(\mathfrak{p})$, $\tau \in \Omega$ a representative of $P$ in the Drinfeld upper half-plane and a Drinfeld modular form $f$, we have
\begin{equation*}
\order_P(f) = \frac{\order_{\tau}(f)}{e_{\tau}}.
\end{equation*}

Then 
\begin{equation*}
\order_P(f_i(z)dz)=\order_P(f_i) + \order_P(dz) = \frac{\order_{\tau}(f_i)}{e_{\tau}}-\frac{e_{\tau}-1}{e_{\tau}},
\end{equation*}
and the result follows.
\end{proof}

\begin{definition}
For any basis $\{f_0, f_1,\ldots f_{g_{\mathfrak{p}}-1}\}$ of $M^{2}_{2,1}(\Gamma_0(\mathfrak{p}))$, we define
\begin{equation*}
W(f_0, \ldots, f_{g_{\mathfrak{p}}-1}) =
\begin{vmatrix}
f_0(z) & D_1(f_0(z)) &\ldots &D_{g_{\mathfrak{p}}-1}(f_0(z))\\
\vdots & & \vdots \\
 f_{g_{\mathfrak{p}}-1}(z) & D_{1}(f_{g_{\mathfrak{p}}-1}(z)) & \ldots & D_{g_{\mathfrak{p}}-1}(f_{g_{\mathfrak{p}}-1}(z))
\end{vmatrix},
\end{equation*}
where $D_n$ is the normalized Hasse derivative introduced in Section \ref{quasimodular}. This is a modular form of weight $g_{\mathfrak{p}}(g_{\mathfrak{p}}+1)$ and type $\frac{g_{\mathfrak{p}}(g_{\mathfrak{p}}+1)}{2}$ for $\Gamma_0(\mathfrak{p})$. 
\end{definition}

If $\{f_0, \ldots f_{g_{\mathfrak{p}}-1}\}$ and $\{f'_0, \ldots f'_{g_{\mathfrak{p}}-1}\}$ are two bases for $M^{2}_{2,1}(\Gamma_0(\mathfrak{p}))$, then $W(f_0, \ldots, f_{g_{\mathfrak{p}}-1})=aW(f'_0, \ldots, f'_{g_{\mathfrak{p}}-1})$ for $0 \neq a \in C$.

\begin{lemma}\label{rationalityintegrality}
There exists a basis $\{f_0, \ldots, f_{g_{\mathfrak{p}}-1}\}$ of $M_{2,1}^2(\Gamma_0(\mathfrak{p}))$ with integral $u$-series coefficients at $\infty$ such that $W(f_0, \ldots, f_{g_{\mathfrak{p}}-1})$ has rational, $\mathfrak{p}$-integral $u$-series coefficients at $\infty$ and
\begin{equation*}
v_{\mathfrak{p}}(W(f_0, \ldots, f_{g_{\mathfrak{p}}-1})) =0.
\end{equation*}
\end{lemma}

\begin{proof}
By Remark \ref{xnotintegrality}, there is a basis $\{f_1, \ldots, f_{g_{\mathfrak{p}}}\}$ for the space $M^{2}_{2,1}(\Gamma_0(\mathfrak{p}))$ that has integral $u$-series coefficients at $\infty$. 

When computing $W(f_1, \ldots, f_{g_{\mathfrak{p}}})$, we will compute $D_n$ for $n\leq g_{\mathfrak{p}}-1$. From the explicit formula (\ref{genusformula}), we have $g_{\mathfrak{p}} \leq 2q^{d-2}$, so that $n\leq 2q^{d-2}-1<q^d$. In this case, Proposition \ref{derivativesok2} says that $D_n$ preserves $\mathfrak{p}$-integrality of the $u$-series coefficients, so $W(f_1, \ldots, f_{g_{\mathfrak{p}}})$ has rational, $\mathfrak{p}$-integral $u$-series coefficients.

Suppose that
\begin{equation*}
v_{\mathfrak{p}}(W(f_0, \ldots, f_{g_{\mathfrak{p}}-1})) >0.
\end{equation*}
Then there exist $a_0, \ldots, a_{g_{\mathfrak{p}}-1}$ with each $a_i \in A$ such that 
\begin{equation*}
a_0 f_0 + \ldots + a_{g_{\mathfrak{p}}-1} f_{g_{\mathfrak{p}}-1} \equiv 0 \pmod{\mathfrak{p}},
\end{equation*}
and for at least one $i$ such that $0 \leq i \leq g_{\mathfrak{p}}-1$,
\begin{equation*}
a_i \not \equiv 0 \pmod{\mathfrak{p}}.
\end{equation*} 
Without loss of generality, suppose that 
\begin{equation*}
a_0 \not \equiv 0 \pmod{\mathfrak{p}}.
\end{equation*} 
Then we have
\begin{equation*}
v_{\mathfrak{p}} \left(f_0 + \frac{1}{a_0}\left( a_1 f_1 + \ldots + a_{g_{\mathfrak{p}}-1} f_{g_{\mathfrak{p}}-1}\right)\right) = m > 0,
\end{equation*}
for some $m \in \mathbb{Z}$. Putting
\begin{equation*}
f_0' = \frac{1}{\pi^m}\left(f_0 + \frac{1}{a_0}\left( a_1 f_1 + \ldots + a_{g_{\mathfrak{p}}-1} f_{g_{\mathfrak{p}}-1}\right)\right),
\end{equation*}
we have that $f_0'$ has integral $u$-series coefficients at $\infty$, $W(f_0', f_1, \ldots f_{g_{\mathfrak{p}}-1})$ has rational, $\mathfrak{p}$-integral $u$-series coefficients at $\infty$, and
\begin{equation*}
v_{\mathfrak{p}}(W(f_0, \ldots, f_{g_{\mathfrak{p}}-1})) >v_{\mathfrak{p}}(W(f_0', \ldots, f_{g_{\mathfrak{p}}-1})).
\end{equation*}

If 
\begin{equation*}
v_{\mathfrak{p}}(W(f_0', \ldots, f_{g_{\mathfrak{p}}-1})) > 0,
\end{equation*}
we may repeat the procedure above. We can continue this process until the valuation is $0$.
\end{proof}

\begin{definition}
As a consequence of Lemma \ref{rationalityintegrality}, there is a unique Drinfeld modular form $W(f_0, \ldots, f_{g_{\mathfrak{p}}-1})$ such that 
\begin{equation*}
v_{\mathfrak{p}}(W(f_0, \ldots, f_{g_{\mathfrak{p}}-1})) =0
\end{equation*}
and the leading coefficient of $W(f_0, \ldots, f_{g_{\mathfrak{p}}-1})$ is a power of $\pi$. We denote this form by $W(z)$ and call it the \emph{modular Wronskian} of $X_0(\mathfrak{p})$. \
\end{definition}

We note that the forms $\{f_0, \ldots, f_{g_{\mathfrak{p}}-1}\}$ which give us $W(z)$ can be chosen to have rational, $\mathfrak{p}$-integral $u$-series coefficients at $\infty$.

We are interested in $W(z)$ because of its relation to the Weierstrass points of $X_0(\mathfrak{p})$:

\begin{theorem}\label{T:wronskian}
Let $(n_1, \ldots, n_{g_{\mathfrak{p}}}) = (1, \ldots, g_{\mathfrak{p}})$ denote the canonical gap sequence of $X_0(\mathfrak{p})$, $P$ be a point of $Y_0(\mathfrak{p})$ and $(n_1(P), \ldots, n_{g_{\mathfrak{p}}}(P))$ be the gap sequence at $P$. Then we have
\begin{equation*}
\order_P(W(z)(dz)^{g_{\mathfrak{p}}(g_{\mathfrak{p}}+1)/2})\geq \sum_{i=1}^{g_{\mathfrak{p}}} (n_i(P)-n_i).
\end{equation*}
In addition, when $P$ is not an elliptic point nor a Weierstrass point, we have equality: 
\begin{equation*}
\order_P(W(z)(dz)^{g_{\mathfrak{p}}(g_{\mathfrak{p}}+1)/2})=0.
\end{equation*}
\end{theorem}

\begin{proof}
Let $P$ be a point on $Y_0(\mathfrak{p})$, and choose a basis $\{f_i\}$ of $M^2_{2,1}(\Gamma_0(\mathfrak{p}))$ that satisfies the conclusion of Proposition \ref{goodbasis}. We also continue to use the notation introduced in the statement of Proposition \ref{goodbasis}. Then 
\begin{equation*}
\order_P(W(f_0, \ldots, f_{g_{\mathfrak{p}}-1})(dz)^{g_{\mathfrak{p}}(g_{\mathfrak{p}}+1)/2})=\order_P(W(z)(dz)^{g_{\mathfrak{p}}(g_{\mathfrak{p}}+1)/2}),
\end{equation*}
so we may work with $W(f_0, \ldots, f_{g_{\mathfrak{p}}-1})$ for convenience.

Choose $\tau \in \Omega$ to be a representative of $P$ in the Drinfeld upper half-plane. By Proposition \ref{orderomega}, for $k= 0, \ldots, g_{\mathfrak{p}}-1$, we have that
\begin{equation*}
\order_{\tau}(D_{k}(f_l)) \geq e_{\tau} j_l(P)+e_{\tau}-1-k
\end{equation*}
with equality if and only if $\binom{e_{\tau}j_l(P)+e_{\tau}-1}{k}\not \equiv 0 \pmod{p}$. When computing the determinant $W(f_0, \ldots, f_{g_{\mathfrak{p}}-1})$, we will be adding terms all of whose order of vanishing at $\tau$ is $\geq \sum_{i=0}^{g_{\mathfrak{p}}-1}(e_{\tau} j_i(P) - i +e_{\tau}-1)$. Thus
\begin{equation*}
\order_{\tau} W(f_0, \ldots, f_{g_{\mathfrak{p}}-1}) \geq \sum_{i=0}^{g_{\mathfrak{p}}-1}(e_{\tau} j_i(P) - i +e_{\tau}-1).
\end{equation*}

Since $X_0(\mathfrak{p})$ has canonical orders $(j_1, \ldots, j_{g_{\mathfrak{p}}-1}) = (1, \ldots, g_{\mathfrak{p}}-1)$ and 
\begin{equation*}
\sum_{i=1}^{g_{\mathfrak{p}}} (n_i(P)-n_i)=\sum_{i=1}^{g_{\mathfrak{p}}-1} (j_i(P)-j_i),
\end{equation*}
for any point $P$ on $X_0(\mathfrak{p})$, we have
\begin{equation*}
\sum_{i=0}^{g_{\mathfrak{p}}-1}(e_{\tau} j_i(P) - i +e_{\tau}-1)= e_{\tau}\sum_{i=1}^{g_{\mathfrak{p}}} (n_i(P)-n_i)+\frac{g_{\mathfrak{p}}(g_{\mathfrak{p}}+1)}{2}(e_{\tau}-1).
\end{equation*}

Thus 
\begin{align*}
\order_P(W(f_0, \ldots, f_{g_{\mathfrak{p}}-1})(dz)^{g_{\mathfrak{p}}(g_{\mathfrak{p}}+1)/2})&\geq \sum_{i=1}^{g_{\mathfrak{p}}} (n_i(P)-n_i)+\frac{g_{\mathfrak{p}}(g_{\mathfrak{p}}+1)}{2}\frac{e_{\tau}-1}{e_{\tau}}-\frac{g_{\mathfrak{p}}(g_{\mathfrak{p}}+1)}{2}\frac{e_{\tau}-1}{e_{\tau}}\\
&= \sum_{i=1}^{g_{\mathfrak{p}}} (n_i(P)-n_i).
\end{align*}.

In the case where $P$ is not elliptic and $P$ is not a Weierstrass point, the terms on the diagonal of $W(f_0, \ldots, f_{g_{\mathfrak{p}}-1})$ have order of vanishing exactly $0$, and all of the terms below the diagonal have order of vanishing strictly greater than $0$. Thus $\order_{\tau} W(f_0, \ldots, f_{g_{\mathfrak{p}}-1}) =0=\sum_{i=1}^{g_{\mathfrak{p}}} (n_i(P)-n_i)$.

\end{proof}

The significance of the previous theorem is that away from the cusps, the divisor 
\begin{equation*}
[W(z)]+ \frac{g_{\mathfrak{p}}(g_{\mathfrak{p}}+1)}{2}[dz]
\end{equation*}
is the modular avatar of the invariant divisor $w$ constructed by St\"{o}hr and Voloch \cite{stohrvoloch}. Consequently, we make the following definition:
\begin{definition}\label{D:weight}
The (modular) Weierstrass weight of a point $P$ on $Y_0(\mathfrak{p})$ is
\begin{equation*}
\wt(P) = \order_P(W(z)(dz)^{g_{\mathfrak{p}}(g_{\mathfrak{p}}+1)/2}).
\end{equation*}
\end{definition}

Finally, to apply Theorem \ref{normtheorem} we will need:

\begin{proposition}\label{T:qodd}
Suppose that $q$ is odd. Then $W(z)$ is an eigenform of the Fricke involution.
\end{proposition}

\begin{proof}
Since we are in odd characteristic, the Fricke involution is diagonalizable. Let $\{f_1, \ldots, f_{g_{\mathfrak{p}}}\}$ be a basis of eigenforms of $W_{\mathfrak{p}}$ of the space $M^{2}_{2,1}(\Gamma_0(\mathfrak{p}))$, say with $f_i | [W_{\mathfrak{p}}] = \lambda_i f_i$.

We compute
\begin{align*}
W(f_0, \ldots, f_{g_{\mathfrak{p}}-1}) \left( \frac{-1}{\pi z}\right) &=
\begin{vmatrix}
f_0\left( \frac{-1}{\pi z}\right) & (D_1f_0) \left( \frac{-1}{\pi z}\right) &\ldots & (D_{g_{\mathfrak{p}}-1}f_0)\left( \frac{-1}{\pi z}\right)  \\
\vdots & & \vdots \\
 f_{g_{\mathfrak{p}}-1}\left( \frac{-1}{\pi z}\right) & (D_{1} f_{g_{\mathfrak{p}}-1})\left( \frac{-1}{\pi z}\right) & \ldots & (D_{g_{\mathfrak{p}}-1} f_{g_{\mathfrak{p}}-1}) \left( \frac{-1}{\pi z}\right)
\end{vmatrix} \\
& = \begin{vmatrix}
\lambda_0 \pi z^2 f_0(z) &  (D_1f_0) \left( \frac{-1}{\pi z}\right)  &\ldots & (D_{g_{\mathfrak{p}}-1}f_0)\left( \frac{-1}{\pi z}\right)   \\
\vdots & & \vdots \\
\lambda_{g_{\mathfrak{p}}-1} \pi z^2 f_{g_{\mathfrak{p}}-1}(z) & (D_{1} f_{g_{\mathfrak{p}}-1})\left( \frac{-1}{\pi z}\right)  & \ldots & (D_{g_{\mathfrak{p}}-1} f_{g_{\mathfrak{p}}-1}) \left( \frac{-1}{\pi z}\right)
\end{vmatrix}. 
\end{align*}

By Proposition \ref{prop:quasimodular}, we have for each $i$ and $n$:
\begin{equation}\label{eqn:der1}
D_n\left( f_i \left( \frac{-1}{\pi z}\right) \right)= z^{-n} \sum_{j=1}^n (-1)^j \binom{n-1}{n-j} \frac{1}{(\pi z)^j} \frac{1}{(-\tilde{\pi})^{n-j}} (D_j f_i) \left(\frac{-1}{\pi z} \right).
\end{equation}

Furthermore using the product rule we have
\begin{equation}\label{eqn:der2}
D_n \left( \lambda_i \pi z^2 f_i(z)\right) = \lambda_i \pi \left( z^2 (D_n f) (z) +2z (D_{n-1} f)(z) + (D_{n-2} f)(z) \right).
\end{equation}

Combining Equations (\ref{eqn:der1}) and (\ref{eqn:der2}) and using induction on $n$, we obtain that
\begin{equation*}
(D_n f_i) \left( \frac{-1}{\pi z}\right) = (-1)^n \lambda_i \pi^{n+1} z^{2n+2} (D_n f_i)(z) + \lambda_i \left( \sum_{j=0}^{n-1} A_{n,j}(\pi, z) (D_n f_i)(z) \right),
\end{equation*}
where $A_{n,j}$ is a polynomial that depends only on $n$ and $j$. Therefore we may successively add to column $C_n$ linear combinations of earlier columns to obtain
\begin{equation*}
W(f_0, \ldots, f_{g_{\mathfrak{p}}-1}) \left( \frac{-1}{\pi z}\right) = \begin{vmatrix} \lambda_i \pi^{n+1} z^{2n+2} (D_n f_i)(z) \end{vmatrix},
\end{equation*}
where $0 \leq i \leq g_{\mathfrak{p}}-1$ indexes the rows and $0 \leq n \leq g_{\mathfrak{p}}-1$ indexes the columns of the matrix.

Pulling out the constant $\lambda_i$ from each row and $\pi^{n+1} z^{2n+2}$ from each column gives
\begin{equation*}
W(f_0, \ldots, f_{g_{\mathfrak{p}}-1})\left( \frac{-1}{\pi z}\right) = \left( \prod_{i=0}^{g_{\mathfrak{p}}-1} \lambda_i \right) \pi^{g_{\mathfrak{p}}(g_{\mathfrak{p}}+1)/2} z^{g_{\mathfrak{p}}(g_{\mathfrak{p}}+1)} W(f_0, \ldots, f_{g_{\mathfrak{p}}-1})(z).
\end{equation*}
Since $W(z)$ is a constant multiple of $W(f_0, \ldots, f_{g_{\mathfrak{p}}-1})(z)$, we conclude that $W(z)$ is an eigenform of the Fricke involution with eigenvalue $\prod_{i=0}^{g_{\mathfrak{p}}-1} \lambda_i$.
\end{proof}

\section{Proof of Theorem \ref{T:main}}\label{consequences}

We are now in a position to prove our main theorem. 

For simplicity throughout this section we will write
\begin{equation}\label{E:wtilde}
\mathcal{W} = \widetilde{N(W)} = \pi^{q^dk/2} \prod_{\gamma \in \Gamma_0(\mathfrak{p})\backslash \Gl_2(A)} W |_{k,l} ,
\end{equation}
which is the form appearing in the statement of Theorem \ref{normtheorem}. It has weight $(q^d+1)g_{\mathfrak{p}}(g_{\mathfrak{p}} +1)$ and type $g_{\mathfrak{p}}(g_{\mathfrak{p}} +1)$. We also write 
\begin{equation}\label{E:poly}
F_{\mathfrak{p}}(x) \defi \prod_{P \in Y_0(\mathfrak{p})} (x- j(P))^{\wt(P)}.
\end{equation}
We note that this is a polynomial since only finitely many points $P$ have $\wt(P) \neq 0$, where $\wt(P)$ is as in Definition \ref{D:weight}, and that we have excluded the cusps from consideration in this product, so that the quantity $j(P)$ is not infinite. 

The strategy to prove Theorem \ref{T:main} is to relate the companion polynomial of $\mathcal{W}(z)$ to the polynomial $F_{\mathfrak{p}}(x)$. Then applying Theorems \ref{tracetheorem} and \ref{normtheorem} to $W(z)$, we show that $\mathcal{W}$ has lower filtration than weight, and conclude that its divisor is supported on all of the supersingular locus.




\begin{theorem}\label{T:epsilon}
Let $\mathcal{W}(z)$ be as in equation (\ref{E:wtilde}). Let $P(\mathcal{W},x)$ be the companion polynomial of the form $\mathcal{W}(z)$ defined in equation (\ref{jpoly}). Then 
\begin{equation*}
P(\mathcal{W},x) = x^{\epsilon(d)} F_{\mathfrak{p}}(x).
\end{equation*}
for 
\begin{equation*}
\epsilon(d) =
\begin{cases}
\frac{1}{q+1}(qg_{\mathfrak{p}}(g_{\mathfrak{p}}+1) -\gamma((q^d+1)g_{\mathfrak{p}}(g_{\mathfrak{p}}+1), g_{\mathfrak{p}}(g_{\mathfrak{p}}+1)))  & \text{if $d$ is even,}\\
\frac{1}{q+1}\gamma((q^d+1)g_{\mathfrak{p}}(g_{\mathfrak{p}}+1), g_{\mathfrak{p}}(g_{\mathfrak{p}}+1)) & \text{if $d$ is odd.}
\end{cases}
\end{equation*}
\end{theorem}

\begin{proof}
Our strategy to relate $P(\mathcal{W},x)$ to $F_{\mathfrak{p}}(x)$ is to first relate the Weierstrass weight $\wt(P)$ of a point to the order of vanishing at $\tau$ of $W(z)$, where $\tau$ is a representative of $P$ in the upper half-plane. We then relate the order of vanishing of $\mathcal{W}(z)$ at $\tau_0 \in \Omega$ to the order of vanishing of $W(z)$ at points $\tau$ that are $\Gl_2(A)$-equivalent to $\tau_0$.

Let $\tau$ be any element of the Drinfeld upper half-plane $\Omega$, and let $P_{\tau}$ be the point on $Y_0(\mathfrak{p})$ corresponding to $\tau$. Further let $e_{\tau}$ be the  order of the stabilizer of $\tau$ in $\widetilde{\Gamma_0}(\mathfrak{p})$. Then we have
\begin{equation*}
\frac{1}{e_{\tau}} \order_{\tau} W(z) = \order_{P_{\tau}} (W(z)(dz)^{g_{\mathfrak{p}}(g_{\mathfrak{p}}+1)/2}) + \frac{g_{\mathfrak{p}}(g_{\mathfrak{p}}+1)}{2} \left( 1- \frac{1}{e_{\tau}} \right).
\end{equation*}

In the case where $P_{\tau}$ is not elliptic, since $e_{\tau} =1$  we simply obtain that
\begin{equation}\label{eq:1notell}
\order_{\tau} W(z) = \wt(P_{\tau}),
\end{equation}
whereas if $P$ is elliptic, in which case $e_{\tau} = q+1$, we have
\begin{equation}\label{eq:1ell}
\order_{\tau} W(z) = (q+1)\wt(P_{\tau}) + q \frac{g_{\mathfrak{p}}(g_{\mathfrak{p}}+1)}{2}.
\end{equation}

We now proceed to the second step of the proof. 

Let first $\tau_0$ be a point in the Drinfeld upper-half space $\Omega$ that is not in the equivalence class of the elliptic point of $X_0(1)$. Since $\mathcal{W}$ is a multiple of 
\begin{equation*}
\prod_{\gamma \in \Gamma_0(\mathfrak{p})\backslash \Gl_2(A)} W |_{k,l} [\gamma],
\end{equation*}
and the map $X_0(\mathfrak{p}) \to X_0(1)$ is unramified above $\tau_0$, we have
\begin{equation}\label{ordernotelliptic}
\order_{\tau_0} \mathcal{W}(z) =\sum_{\substack{P_{\tau} \in Y_0(\mathfrak{p}),\\ \tau \sim \tau_0 }} \order_{\tau} W(z) =  \sum_{\substack{P_{\tau} \in Y_0(\mathfrak{p}),\\ \tau \sim \tau_0 }} \wt(P_{\tau}),
\end{equation}
where $\sim$ denotes $\Gl_2(A)$-equivalence.

We note that in equation (\ref{ordernotelliptic}), the left-hand side is exactly the power of $(x - j(\tau_0))$ appearing in $P(\mathcal{W},x)$ and the right-hand side is exactly the power of $(x - j(\tau_0))$ in $x^{\epsilon(d)}F_{\mathfrak{p}}(x)$.

We now consider the case of $\tau_0$ in the equivalence class of the elliptic point of $X_0(1)$, i.e. $j(\tau_0) = 0$.

\noindent
\underline{The case of $d$ even} If the degree $d$ of the prime polynomial generating $\mathfrak{p}$ is even, then $X_0(\mathfrak{p})$ has two elliptic points, both of which are unramified over $X_0(1)$. The fiber above the elliptic point of $X_0(1)$ in $X_0(\mathfrak{p})$ contains in addition $\frac{q^d-1}{q+1}$ non-elliptic points, each ramified above $P_{\tau_0} \in X_0(1)$ with index $q+1$ \cite[pages 77-78]{gekeler11}. 

Thus if $\tau_0 \in \Omega$ is in the $\Gl_2(A)$-equivalence class of the elliptic point on $X_0(1)$, using equations (\ref{eq:1notell}) and (\ref{eq:1ell}), we have
\begin{equation*}
\order_{\tau_0} \mathcal{W}(z) = 2 q \frac{g_{\mathfrak{p}}(g_{\mathfrak{p}}+1)}{2} +(q+1) \sum_{\substack{\tau \in Y_0(\mathfrak{p}),\\ \tau \sim \tau_0 }} \wt(P_{\tau}).
\end{equation*}

On the other hand, by equation (\ref{jpoly}), we have
\begin{equation*}
\order_{\tau_0} \mathcal{W}(z)= \gamma((q^d+1)g_{\mathfrak{p}}(g_{\mathfrak{p}}+1), g_{\mathfrak{p}}(g_{\mathfrak{p}}+1)) + (q+1)M,
\end{equation*}
where $M$ is the order of vanishing of $P(\mathcal{W},x)$ at $j(\tau_0)=0$.

Combining these two equations we obtain 
\begin{equation}\label{ordereven}
M = \frac{1}{q+1}(qg_{\mathfrak{p}}(g_{\mathfrak{p}}+1) -\gamma((q^d+1)g_{\mathfrak{p}}(g_{\mathfrak{p}}+1), g_{\mathfrak{p}}(g_{\mathfrak{p}}+1))) + \sum_{\substack{\tau \in Y_0(\mathfrak{p}),\\ \tau \sim \tau_0 }} \wt(P_{\tau}).
\end{equation}
For $d$ even, let $\epsilon(d) = \frac{1}{q+1}(qg_{\mathfrak{p}}(g_{\mathfrak{p}}+1) -\gamma((q^d+1)g_{\mathfrak{p}}(g_{\mathfrak{p}}+1), g_{\mathfrak{p}}(g_{\mathfrak{p}}+1))) $.

Equations (\ref{ordernotelliptic}) and (\ref{ordereven}) imply the equality of polynomials
\begin{equation*}
P(\mathcal{W},x) = x^{\epsilon(d)} F_{\mathfrak{p}}(x).
\end{equation*}

\noindent
\underline{The case of $d$ odd} If the degree $d$ of the prime polynomial generating $\mathfrak{p}$ is odd, then $X_0(\mathfrak{p})$ has no elliptic points, and the fiber above the elliptic point of $X_0(1)$ in $X_0(\mathfrak{p})$ contains $\frac{q^d+1}{q+1}$ non-elliptic points, each ramified above $X_0(1)$ with index $q+1$. Thus if $\tau_0$ is in the $\Gl_2(A)$-equivalence class of the elliptic point on $X_0(1)$, we have
\begin{equation*}
\order_{\tau_0} \mathcal{W}(z)= (q+1) \sum_{\substack{\tau \in Y_0(\mathfrak{p}),\\ \tau \sim \tau_0 }} \wt(P_{\tau}).
\end{equation*}

On the other hand, by equation (\ref{jpoly}), we have
\begin{equation*}
\order_{\tau_0} \mathcal{W}(z) = \gamma((q^d+1)g_{\mathfrak{p}}(g_{\mathfrak{p}}+1), g_{\mathfrak{p}}(g_{\mathfrak{p}}+1)) + (q+1)M,
\end{equation*}
where $M$ is the order of vanishing of $P(\mathcal{W},x)$ at $j(\tau_0)=0$.

Combining these two equations we obtain that 
\begin{equation}\label{orderodd}
M = \frac{1}{q+1}\gamma((q^d+1)g_{\mathfrak{p}}(g_{\mathfrak{p}}+1), g_{\mathfrak{p}}(g_{\mathfrak{p}}+1)) + \sum_{\substack{\tau \in Y_0(\mathfrak{p}),\\ \tau \sim \tau_0 }} \wt(P_{\tau}).
\end{equation}
For $d$ odd, let $\epsilon(d) = \frac{1}{q+1}\gamma((q^d+1)g_{\mathfrak{p}}(g_{\mathfrak{p}}+1), g_{\mathfrak{p}}(g_{\mathfrak{p}}+1)) $.

Equations (\ref{ordernotelliptic}) and (\ref{orderodd}) now imply
\begin{equation*}
P(\mathcal{W},x) = x^{\epsilon(d)}F_{\mathfrak{p}}(x),
\end{equation*}
as in the even case, but with a different $\epsilon(d)$.

\end{proof}

We now use the trace map to obtain a form of low weight for $\Gl_2(A)$ that is congruent to $W(z)$ modulo $\mathfrak{p}$.

\begin{theorem}\label{gl2wronskian}
Let $q\geq 3$, then there exists a Drinfeld modular form $F$ of weight $g_{\mathfrak{p}}(g_{\mathfrak{p}}+q^d)$ and type $\frac{g_{\mathfrak{p}}(g_{\mathfrak{p}}+1)}{2}$ for $\Gl_2(A)$ such that 
\begin{equation*}
W(z) \equiv F(z) \pmod{\mathfrak{p}}.
\end{equation*}
\end{theorem}

\begin{proof}
We choose a basis $\{f_0, \ldots, f_{g_{\mathfrak{p}}-1}\}$ for the space $M^{2}_{2,1}(\Gamma_0(\mathfrak{p}))$ such that 
\begin{equation*}
W(z) = W(f_0, \ldots, f_{g_{\mathfrak{p}}-1}),
\end{equation*}
and such that for each $i$ $f_i$ has rational, $\mathfrak{p}$-integral $u$-series coefficients at $\infty$.

By Theorem \ref{tracetheorem}, there is a basis $\{F_0, \ldots, F_{g_{\mathfrak{p}}-1}\}$ for the space $M^{2}_{q^d+1,1}(\Gl_2(A))$, all of whose elements have rational, $\mathfrak{p}$-integral $u$-series coefficients and such that $f_i \equiv F_i \pmod{\mathfrak{p}}$. 

As we remarked in the proof of Proposition \ref{rationalityintegrality}, when computing the forms $W(f_0, \ldots, f_{g_{\mathfrak{p}}-1})$ and $W(F_0, \ldots, F_{g_{\mathfrak{p}}-1})$, one needs to compute $D_n$ for $n<q^d$. Thus in all of the cases we will consider, we have that $f_i\equiv F_i \pmod{\mathfrak{p}}$ implies that $D_n(f_i) \equiv D_n(F_i) \pmod{\mathfrak{p}}$ by Corollary \ref{derivativesok}.

Therefore we have 
\begin{equation*}
W(f_0, \ldots, f_{g_{\mathfrak{p}}-1}) \equiv W(F_0, \ldots, F_{g_{\mathfrak{p}}-1})\pmod{\mathfrak{p}}.
\end{equation*}

The form $W(F_0, \ldots, F_{g_{\mathfrak{p}}-1})$ is modular for $\Gl_2(A)$ of weight $g_{\mathfrak{p}}(g_{\mathfrak{p}}+q^d)$ and type $\frac{g_{\mathfrak{p}}(g_{\mathfrak{p}}+1)}{2}$, and we denote it $F$ for simplicity.
\end{proof}



We can now prove Theorem \ref{T:main}:

\begin{proof}[Proof of Theorem \ref{T:main}]
Since $W$ has rational, $\mathfrak{p}$-integral $u$-series coefficients at $\infty$ and is an eigenform of the Fricke involution, Theorem \ref{normtheorem} states that
\begin{equation*}
\mathcal{W} = \widetilde{\N(W)} \equiv W^2 \pmod{\mathfrak{p}},
\end{equation*}
As remarked earlier, $\mathcal{W}$ is a form of weight $(q^d+1)g_{\mathfrak{p}}(g_{\mathfrak{p}}+1)$ and type $g_{\mathfrak{p}}(g_{\mathfrak{p}}+1)$ for $\Gl_2(A)$.

By Theorem \ref{gl2wronskian}, we have
\begin{equation}\label{E:congruence}
\mathcal{W} \equiv F^2 \pmod{\mathfrak{p}}.
\end{equation}

The form $F^2$ is of weight $2g_{\mathfrak{p}}(g_{\mathfrak{p}}+q^d)$ and type $g_{\mathfrak{p}}(g_{\mathfrak{p}}+1)$.

We note now that the proof of Proposition \ref{dobiwagewang} can be adapted say the following: Let $f$ and $f'$ be two Drinfeld modular forms for $\Gl_2(A)$ of weights $k > k'$ and of types $l$ and $l'$, respectively, both with rational $\mathfrak{p}$-integral $u$-series coefficients and not $\equiv 0 \pmod{\mathfrak{p}}$. Then for $\alpha = \frac{k - k'}{q^d-1}$ and $a = \lfloor \frac{\alpha \gamma(q^d-1,0)q + \gamma(k,l)}{q+1}\rfloor$, the polynomial $x^aP(f,x)$ is divisible by $S_{\mathfrak{p}}(x)^{\alpha}$ in $\mathbb{F}_{\mathfrak{p}}[x]$. (We recall that $\mathbb{F}_{\mathfrak{p}}$ is the field $A/\mathfrak{p}$.)

Applying this to equation (\ref{E:congruence}), we have $\alpha = g_{\mathfrak{p}}(g_{\mathfrak{p}}-1)$. Then in $\mathbb{F}_{\mathfrak{p}}[x]$ we have that 
\begin{equation*}
S_{\mathfrak{p}}(x)^{g_{\mathfrak{p}}(g_{\mathfrak{p}}-1)} \mid x^{a} P(\mathcal{W},x) = x^{a+\epsilon(d)} F_{\mathfrak{p}}(x),
\end{equation*} 
where 
\begin{equation}\label{eq:a}
a = \left\lfloor \frac{g_{\mathfrak{p}}(g_{\mathfrak{p}}-1)\gamma(q^d-1,0)q + \gamma((q^d+1)g_{\mathfrak{p}}(g_{\mathfrak{p}}+1),g_{\mathfrak{p}}(g_{\mathfrak{p}}+1))}{q+1}\right\rfloor .
\end{equation}

\noindent
\underline{The case of $d$ even} In this case $j =0 $ is not supersingular at $\mathfrak{p}$, so $x$ does not divide $S_{\mathfrak{p}}(x)$, and we conclude that 
\begin{equation*}
S_{\mathfrak{p}}(x)^{g_{\mathfrak{p}}(g_{\mathfrak{p}}-1)} \mid F_{\mathfrak{p}}(x).
\end{equation*} 
Thus each supersingular $j$-invariant is the reduction modulo $\mathfrak{p}$ of a root of $F_{\mathfrak{p}}(x)$.

By Theorem \ref{T:wronskian}, for $P \in Y_0(\mathfrak{p})$,
\begin{equation*}
\wt(P) = \order_P(W(z)(dz)^{g_{\mathfrak{p}}(g_{\mathfrak{p}}+1)/2})\geq \sum_{i=1}^{g_{\mathfrak{p}}} (n_i(P)-n_i),
\end{equation*}
with $\wt(P)=0$ if $P$ is neither a Weierstrass point nor an elliptic point. Recall also that a Weierstrass point is a point such that 
\begin{equation*}
\sum_{i=1}^{g_{\mathfrak{p}}} (n_i(P)-n_i) >0.
\end{equation*}

By definition (equation (\ref{E:poly})) the polynomial $F_{\mathfrak{p}}(x)$ has zeroes at the Weierstrass points, and possibly also at the elliptic points of $X_0(\mathfrak{p})$, which have $j=0$. Since $j=0$ is not supersingular when $d$ is even, then each supersingular $j$-invariant is the reduction modulo $\mathfrak{p}$ of the $j$-invariant of a Weierstrass point. 
\newline

\noindent
\underline{The case of $d$ odd} As argued in the case of $d$ even, the zeroes of $F_{\mathfrak{p}}$ are either Weierstrass points or elliptic points. Since $X_0(\mathfrak{p})$ does not have elliptic points when $d$ is odd, the zeroes of $F_{\mathfrak{p}}$ are exactly the Weierstrass points.

Since 
\begin{equation*}
S_{\mathfrak{p}}(x)^{g_{\mathfrak{p}}(g_{\mathfrak{p}}-1)} \mid  x^{a+\epsilon(d)} F_{\mathfrak{p}}(x),
\end{equation*} 
where $a$ is as in equation (\ref{eq:a}) and $\epsilon(d)$ is as in the statement of Theorem \ref{T:epsilon}, we conclude that each supersingular $j$-invariant in characteristic $\mathfrak{p}$ except possibly $j=0$ is the reduction modulo $\mathfrak{p}$ of the $j$-invariant of a Weierstrass point.

To conclude that $j=0$ is also the $j$-invariant of a Weierstrass point, we must show that
\begin{equation*}
g_{\mathfrak{p}}(g_{\mathfrak{p}}-1) > a+\epsilon(d),
\end{equation*}
from which it will follow that $x \mid F_{\mathfrak{p}}(x)$. 

We first investigate the number $\epsilon(d) = \frac{1}{q+1}\gamma((q^d+1)g_{\mathfrak{p}}(g_{\mathfrak{p}}+1), g_{\mathfrak{p}}(g_{\mathfrak{p}}+1))$. Since $(q^d+1)g_{\mathfrak{p}}(g_{\mathfrak{p}}+1)$ is divisible by $q+1$ and by the uniqueness of the numbers $\gamma((q^d+1)g_{\mathfrak{p}}(g_{\mathfrak{p}}+1), g_{\mathfrak{p}}(g_{\mathfrak{p}}+1))$ and $\mu((q^d+1)g_{\mathfrak{p}}(g_{\mathfrak{p}}+1), g_{\mathfrak{p}}(g_{\mathfrak{p}}+1))$, satisfying the conditions of (\ref{eq:uniqueness}), we must have
\begin{equation*}
\mu((q^d+1)g_{\mathfrak{p}}(g_{\mathfrak{p}}+1), g_{\mathfrak{p}}(g_{\mathfrak{p}}+1)) = \frac{(q^d+1)g_{\mathfrak{p}}(g_{\mathfrak{p}}+1)}{q+1}
\end{equation*}
and
\begin{equation*}
\gamma((q^d+1)g_{\mathfrak{p}}(g_{\mathfrak{p}}+1), g_{\mathfrak{p}}(g_{\mathfrak{p}}+1)) =0,
\end{equation*}
so $\epsilon(d) = 0$ when $d$ is odd.

Since $d$ is odd, we have that $\gamma(q^d-1,0)=1$ and in light of the work above, the formula for $a$ simplifies to 
\begin{equation*}
a = \left\lfloor \frac{g_{\mathfrak{p}}(g_{\mathfrak{p}}-1)q}{q+1}\right\rfloor.
\end{equation*}

Since 
\begin{equation*}
\left\lfloor \frac{g_{\mathfrak{p}}(g_{\mathfrak{p}}-1)q}{q+1}\right\rfloor \leq  \frac{g_{\mathfrak{p}}(g_{\mathfrak{p}}-1)q}{q+1} < g_{\mathfrak{p}}(g_{\mathfrak{p}}-1),
\end{equation*}
it follows that $j=0$ is also the reduction modulo $\mathfrak{p}$ of the $j$-invariant of a Weierstrass point of $X_0(\mathfrak{p})$.
\end{proof}



\subsection{A refinement of the statement}\label{refinement}

Since $\mathcal{W}$ is of weight $(q^d+1)g_{\mathfrak{p}}(g_{\mathfrak{p}}+1)$ and type $g_{\mathfrak{p}}(g_{\mathfrak{p}}+1)$, $F^2$ is of weight $2g_{\mathfrak{p}}(g_{\mathfrak{p}}+q^d)$ and type $g_{\mathfrak{p}}(g_{\mathfrak{p}}+1)$, and
\begin{equation*}
\mathcal{W} \equiv F^2 \pmod{\mathfrak{p}},
\end{equation*}
we have that $\mathcal{W}$ and $F^2 g_d^{g_{\mathfrak{p}}(g_{\mathfrak{p}}-1)}$ are two forms of the same weight and type that are congruent modulo $\mathfrak{p}$, and therefore their companion polynomials are congruent modulo $\mathfrak{p}$:
\begin{equation*}
P(\mathcal{W},x) \equiv P(F^2 g_d^{g_{\mathfrak{p}}(g_{\mathfrak{p}}-1)},x) \pmod{\mathfrak{p}}.
\end{equation*}
\newline 

\noindent
\underline{The case of $d$ even} Applying Proposition \ref{P:companionpolys} part \ref{evens} $g_{\mathfrak{p}}(g_{\mathfrak{p}}-1)$times, we have
\begin{equation*}
P(\mathcal{W},x) \equiv P(F^2,x) P(g_d,x)^{g_{\mathfrak{p}}(g_{\mathfrak{p}}-1)} \pmod{\mathfrak{p}}.
\end{equation*}

Since $P(\mathcal{W},x) = x^{\epsilon(d)} F_{\mathfrak{p}}(x)$ and $ P(g_d,x) = S_{\mathfrak{p}}(x)$, we have
\begin{equation*}
x^{\epsilon(d)} F_{\mathfrak{p}}(x) \equiv P(F^2,x) S_{\mathfrak{p}}(x)^{g_{\mathfrak{p}}(g_{\mathfrak{p}}-1)} \pmod{\mathfrak{p}}.
\end{equation*}
Therefore the extent to which we can understand the polynomial $P(F^2,x)$ will determine how much more we can understand about the Weierstrass points of $X_0(\mathfrak{p})$ and the quantity $\wt(P)$ defined in this paper. In addition, it is this polynomial which keeps us from obtaining the main result of \cite{ahlgrenono} in full generality in this setting.
\newline

\noindent
\underline{The case of $d$ odd} Applying Proposition \ref{P:companionpolys} part \ref{odds} $g_{\mathfrak{p}}(g_{\mathfrak{p}}-1)$times, we have
\begin{equation*}
P(\mathcal{W},x) \equiv x^{b} P(F^2,x) P(g_d,x)^{g_{\mathfrak{p}}(g_{\mathfrak{p}}-1)} \pmod{\mathfrak{p}},
\end{equation*}
where $b = \lfloor \frac{g_{\mathfrak{p}}(g_{\mathfrak{p}}-1) + \gamma(k,l)}{q+1}\rfloor$. 

Then we have 
\begin{equation*}
F_{\mathfrak{p}}(x) \equiv x^{b} P(F^2,x) P(g_d,x)^{g_{\mathfrak{p}}(g_{\mathfrak{p}}-1)} \pmod{\mathfrak{p}},
\end{equation*}
since $\epsilon(d)=0$ when $d$ is odd.



\section{The order of vanishing of $W(z)$ at the cusps}\label{divisorcusp}

In the discussion surrounding the definition of modular weight (Definition \ref{D:weight}), we avoided considering the valuation of the divisor 
\begin{equation*}
[W(z)]+ \frac{g_{\mathfrak{p}}(g_{\mathfrak{p}}+1)}{2}[dz]
\end{equation*}
at the two cusps of $X_0(\mathfrak{p})$. From the algebraic theory of Weierstrass points developed in Section \ref{theory}, we would expect this divisor to have valuation 0 or at worst positive valuation at the cusps. Unfortunately at present we cannot show this directly, but we proceed to say what we can.



We begin by consider the divisor of $dz$ at the cusps. From explicit computations \cite{gekelerjacobian}, we have that $\frac{1}{u^2}du=-\tilde{\pi}dz$. Recall from Section \ref{infinity} the function $t=u^{q-1}$, which is a uniformizer at the cusps $0$ and $\infty$ for $X_0(\mathfrak{p})$. Then we have 
\begin{equation*}
\frac{1}{t^{q/(q-1)}}dt=\tilde{\pi}dz,
\end{equation*}
and $dz$ has a pole of order 
\begin{equation*}
\frac{q}{q-1}
\end{equation*}
at the cusps $0$ and $\infty$. 

\begin{proposition}\label{goodbasiscusp}
Let $P$ be a cusp of $X_0(\mathfrak{p})$, and write  $\tau=0$ or $\tau=\infty$. Then there is a basis $\{f_i\}_{i=0}^{g_{\mathfrak{p}}-1}$ of $M^2_{2,1}(\Gamma_0(\mathfrak{p}))$ such that:
\begin{equation*}
\order_{\tau}(f_i)= (q-1)i+ q 
\end{equation*}
for each $i$.
\end{proposition}

\begin{proof}
As in the proof of Proposition \ref{goodbasis}, since the canonical orders at $P$ are $(1,\ldots, g_{\mathfrak{p}}-1)$ (recall that the cusps are not Weierstrass points) we have that there is a basis of $M^2_{2,1}(\Gamma_0(\mathfrak{p}))$ with
\begin{equation*}
\order_P(f_i(z)dz)=i.
\end{equation*}

If $P$ is a cusp of $X_0(\mathfrak{p})$, $\tau=0$ or $\infty$, and $f$ is a Drinfeld modular form for $\Gamma_0(\mathfrak{p})$, we have
\begin{equation*}
\order_P(f) = \frac{\order_{\tau}(f)}{q-1}.
\end{equation*}

Then since
\begin{equation*}
\order_P(f_i(z)dz)=\order_P(f_i)+\order_P(dz) = \frac{\order_{\tau}(f_i(z))}{q-1}-\frac{q}{q-1},
\end{equation*}
the result follows.
\end{proof}

For the next result we will need the following definition: Let $n$ be a positive integer and $q$ be a power of a prime such that the expansion of $n$ in base $q$ is $n = \sum_{i=0}^r n_iq^i$, where each $0 \leq n_i \leq q-1$ for each $i$. Then we write $\| n\|_q = \sum_{i=0}^r n_i$. 

\begin{proposition}\label{ordercusp} 
Let $f$ be analytic at $\infty$, then $\order_{\infty} D_n(f) \geq \order_{\infty}(f) + \| n \|_q$. 
\end{proposition}

\begin{proof}
Let
\begin{equation*}
\alpha_{n,j} = \sum_{\substack{n_1, \ldots, n_j \geq 0 \\ q^{n_1}+ \cdots +q^{n_j}=n}} \frac{1}{d_{n_1} \cdots d_{n_j}},
\end{equation*}
where $d_i$ was defined at the beginning of Section \ref{quasimodular}. Then we have that $\alpha_{n,j} \neq 0$ if and only if $j \equiv \| n \|_q \pmod{q-1}$ and $j \leq n$. Indeed, the least $j$ such that there exists $n_1, \ldots n_j \geq 0$ with $q^{n_1}+ \cdots +q^{n_j}=n$ is $\| n \|_q$. Furthermore, given a tuple $(n_1, \ldots, n_j)$ such that $q^{n_1}+ \cdots +q^{n_j}=n$ and at least one $n_i >0$, we can write another tuple $(m_1, \ldots, m_{j+q-1})$ such that $q^{m_1}+ \cdots +q^{m_{j+q-1}}=n$ by ``unbundling" a term $q^{n_i}$ into $q$ terms of the form $q^{n_i-1}$ if $n_i>0$. This process is no longer possible when each $n_i=0$, in which case we have $q^0 + \ldots +q^0 = n$. This shows that for each $j$ between $\| n \|_q$ and $n$ such that $j \equiv \| n \|_q \pmod{q-1}$, $\alpha_{n,j} \neq 0$. Conversely if there is $(n_1, \ldots, n_j)$ such that $q^{n_1}+ \cdots +q^{n_j}=n$, then
\begin{equation*}
n  = (q^{n_1}-1)+ \cdots +(q^{n_j}-1) + j \equiv j \pmod{q-1}.
\end{equation*}
But applying this same trick to the sum $n = \sum_{i=0}^r n_iq^i$, we have $n \equiv \| n \|_q \pmod{q-1}$.

Using the explicit formula given in Proposition \ref{formuladn}, we have that if $f = \sum_{i=0}^{\infty} a_i u^i$ and $D_n f = \sum_{i=0}^{\infty} b_{n,i}u^i$, then
\begin{equation*}
b_{n,i} = \sum_{r = 1}^{i-1}(-1)^{n+r} \binom{i-1}{r} \alpha_{n,r} a_{i-r}.
\end{equation*}
In light of the remarks above, the only terms that can possibly appear in this sum are those with $r \equiv \| n \|_q \pmod{q-1}$. Therefore the least $i$ for which $b_{n,i}$ is possibly nonzero is one where $i - \| n \|_q \geq \order_{\infty}(f)$.
\end{proof}

\begin{proposition}\label{goodordercusp}
Let $\mathfrak{p}$ be generated by a prime polynomial of degree 3, so that $g_{\mathfrak{p}} = q$. Then if $P$ is the cusp $\infty$ of $X_0(\mathfrak{p})$, we have
\begin{equation*}
\order_P(W(z)(dz)^{g_{\mathfrak{p}}(g_{\mathfrak{p}}+1)/2}) \geq 0.
\end{equation*}
\end{proposition}

\begin{proof}
We choose a basis $\{f_i\}$ of $M^2_{2,1}(\Gamma_0(\mathfrak{p}))$ that satisfies the conclusion of Proposition \ref{goodbasiscusp} at $\infty$. Then 
\begin{equation*}
\order_P(W(f_0, \ldots, f_{g_{\mathfrak{p}}-1})(dz)^{g_{\mathfrak{p}}(g_{\mathfrak{p}}+1)/2})=\order_P(W(z)(dz)^{g_{\mathfrak{p}}(g_{\mathfrak{p}}+1)/2}),
\end{equation*}
so we may work with $W(f_0, \ldots, f_{g_{\mathfrak{p}}-1})$ for convenience.

By Proposition \ref{ordercusp}, for $k= 0, \ldots, g_{\mathfrak{p}}-1=q-1$, we have that
\begin{equation*}
\order_{\infty}(D_{k}(f_l)) \geq (q-1)l+q+\| k\|_q = (q-1)l +q +k,
\end{equation*}
since $\| k\|_q = k$ because $0 \leq k \leq q-1$. When computing the determinant $W(f_0, \ldots, f_{g_{\mathfrak{p}}-1})$, we will be adding terms all of whose order of vanishing at $\infty$ is $\geq \sum_{i=0}^{g_{\mathfrak{p}}-1}((q-1)i+q+i)$. Thus
\begin{equation*}
\order_{\tau} W(f_0, \ldots, f_{g_{\mathfrak{p}}-1}) \geq \sum_{i=0}^{g_{\mathfrak{p}}-1}q(i+1).
\end{equation*}
We have
\begin{equation*}
\sum_{i=0}^{g_{\mathfrak{p}}-1}q(i+1) = q \frac{g_{\mathfrak{p}}(g_{\mathfrak{p}}+1)}{2}.
\end{equation*}


And so
\begin{equation*}
\order_P(W(f_0, \ldots, f_{g_{\mathfrak{p}}-1})(dz)^{g_{\mathfrak{p}}(g_{\mathfrak{p}}+1)/2}) \geq \frac{q}{q-1} \frac{g_{\mathfrak{p}}(g_{\mathfrak{p}}+1)}{2} -  \frac{q}{q-1} \frac{g_{\mathfrak{p}}(g_{\mathfrak{p}}+1)}{2} =0.
\end{equation*}
\end{proof}

\begin{remark}
To obtain Proposition \ref{goodordercusp} for all $\mathfrak{p}$, it would be sufficient to show that for a basis of $M^2_{2,1}(\Gamma_0(\mathfrak{p}))$ satisfying the conclusion of Proposition \ref{goodbasiscusp} at $\infty$,
\begin{equation*}
\order_{\infty}(D_{k}(f_l)) \geq \order_{\infty}(f_l)+k = (q-1)l +q +k,
\end{equation*}
but that is not true. For example, fixing $q=3$ and any $d >3$, we have that $\order_{\infty} (f_1) = 5$ but
\begin{equation*}
\order_{\infty}(D_3(f_1)) = 6 < 5+3 = 8.
\end{equation*}
For this reason we expect that to show that the divisor of $W(z)(dz)^{g_{\mathfrak{p}}(g_{\mathfrak{p}}+1)/2}$ is effective at the cusps will require an intricate and precise study of the action of $D_n$, beyond the scope of what we wish to accomplish in this paper.
\end{remark}

\begin{remark}
We note that it should be straightforward to obtain a result similar to Proposition \ref{goodordercusp} for the cusp $0$ using Lemma \ref{prop:quasimodular}, but we do not need it at the moment.
\end{remark}

\section{A special case}\label{specialcase}

As remarked in Section \ref{refinement}, because of its significance it would be of great interest to compute the reduction modulo $\mathfrak{p}$ of the form $F$ explicitly, or even just its divisor modulo $\mathfrak{p}$. This task, however, involves computing the action of $D_n$ for large $n$, which quickly gets complicated. However, under some rather restrictive conditions we are able to prove Theorem \ref{T:computation} which provides an explicit form which is congruent to $F$ modulo $\mathfrak{p}$, and gives us an analogue of the main theorem of \cite{rohrlich2}. This in turns allows us to prove Theorem \ref{T:ahlgrenono}, which is an analogue of the main theorem of \cite{ahlgrenono}.

We will need some notation: For a system of derivatives $\{ \delta_n \}$ which is a higher derivation, and a positive integer $n$, we will write $W_{\delta}(f_1, \ldots, f_{n}) $ for the quantity
\begin{equation*}
\begin{vmatrix}
f_1 & \delta_{1}(f_1) &\ldots &\delta_{n-1}(f_1)\\
\vdots & & \vdots \\
 f_{n} & \delta_{1}(f_{g}) & \ldots & \delta_{n-1}(f_{n})
\end{vmatrix}.
\end{equation*}
We note that $W_{D}(f_1, \ldots, f_{n}) = W(f_1, \ldots, f_{n})$.

Recall from the proof of Theorem \ref{gl2wronskian} that there exists a basis $\{F_0, \ldots, F_{g_{\mathfrak{p}}-1}\}$ for the space $M^{2}_{q^d+1,1}(\Gl_2(A))$, all of whose elements have rational, $\mathfrak{p}$-integral $u$-series coefficients and such that
\begin{equation}\label{wronskianequiv}
W(z) \equiv W_{D}(F_0, \ldots, F_{g_{\mathfrak{p}}-1})\pmod{\mathfrak{p}}.
\end{equation}
Furthermore, $W_D(F_0, \ldots, F_{g_{\mathfrak{p}}-1})$ was the form which we denoted by $F$.

Let $\partial_n^{(d)}$ be the Serre operator from Section \ref{tool}, we have that $D_n(f)$ and $\partial_n^{(k)}(f)$, for $k$ the weight of $f$, differ by the sum
\begin{equation*}
\sum_{i=1}^{n}(-1)^i\binom{k+n-1}{i}(D_{i-1}E)(D_{n-i}f).
\end{equation*}
We note that the quantity $(-1)^i\binom{k+n-1}{i}(D_{i-1}E)$ depends on $k$ and $n$, but not on $f$. To ease notation, we write $M_{D}$ for the matrix appearing in the definition of $W_D(F_0, \ldots, F_{g_{\mathfrak{p}}-1})$, and $M_{\partial}$ for the matrix appearing in the definition of $W_{\partial}(F_0, \ldots, F_{g_{\mathfrak{p}}-1})$. Then we have that the $(n+1)$st column of $M_{\partial}$ is equal to the $(n+1)$st column of $M_{D}$ plus a linear combination of earlier columns of $M_{D}$. Since we are taking a determinant, we conclude that 
\begin{equation}\label{serreequation}
W_D(F_0, \ldots, F_{g_{\mathfrak{p}}-1})=W_{\partial}(F_0, \ldots, F_{g_{\mathfrak{p}}-1}).
\end{equation}

In order to proceed with the computation, we first restrict our attention to the case where $d=3$. In that case $g_{\mathfrak{p}}=q$ and the canonical orders of $X_0(\mathfrak{p})$ are $(1, \ldots, q-1)$.

We now give a basis for the space $M^{2}_{q^3+1,1}(\Gl_2(A))$. We recall that the algebra of Drinfeld modular forms for $\Gl_2(A)$ is generated by $g$, a Drinfeld modular form of weight $q-1$ and type 0 which is not a cusp form, and $h$, a Drinfeld modular form of weight $q+1$ and type 1 with a simple zero at the cusp. We note that both $g$ and $h$ have integral $u$-series coefficients at $\infty$. To give a basis for $M^{2}_{q^3+1,1}(\Gl_2(A))$ with integral $u$-series coefficients is thus simply equivalent to enumerating all monomials $g^ah^b$ with $a \geq 0$, $b \geq 2$ and such that
\begin{equation*}
a(q-1) + b(q+1) = q^3+1
\end{equation*}
and $b \equiv 1 \pmod{q-1}$.
This is easily done and we get that
\begin{equation*}
g^{n(q+1)}h^{q^2-q+1-n(q-1)}, \qquad 0\leq n \leq q-1
\end{equation*}
is a basis of Drinfeld modular forms with integral $u$-series coefficients for the space we are interested in. 

Therefore there is a constant $a\in K$ such that
\begin{align*}
W_{D}(F_0, \ldots, F_{g_{\mathfrak{p}}-1}) & =W_{\partial}(F_0, \ldots, F_{g_{\mathfrak{p}}-1})\\
& = a W_{\partial}(h^{q^2-q+1}, \ldots, g^{q^2-1}h^q),
\end{align*}
where the first equality is equation (\ref{serreequation})  and so
\begin{equation}\label{explicitcong}
W(z) \equiv a W_{\partial}(h^{q^2-q+1}, \ldots, g^{q^2-1}h^q) \pmod{\mathfrak{p}}
\end{equation}
by equation (\ref{wronskianequiv}).


As before we make the convention that if $f$ is a Drinfeld modular form of weight $k$, then $\partial (f)=\partial_1^{(k)}(f)$. Then if $1 \leq n < p$ for $p$ odd, we have $\partial^n f = n! \partial_n^{(k)} f$, where as before the exponent of $n$ on $\partial$ denotes the $n$-fold iteration. Therefore when $q=p$, the computation of $W_{\partial}(h^{p^2-p+1}, \ldots, g^{p^2-1}h^p)$ can be performed using the fact that $\partial(g)=-h$ and $\partial(h)=0$, and we get
\begin{equation*}
W_{\partial}(h^{p^2-p+1}, \ldots, g^{p^2-1}h^p)=g^{\frac{p^2(p-1)}{2}}h^{\frac{p^2(p+1)}{2}}.
\end{equation*}

Thus equation (\ref{explicitcong}) becomes
\begin{equation}\label{congeq}
W(z) \equiv a g^{\frac{p^2(p-1)}{2}}h^{\frac{p^2(p+1)}{2}} \pmod{\mathfrak{p}}
\end{equation}


We now investigate the value of the constant $a$. The first non-zero $u$-series coefficient of $g^{\frac{p^2(p-1)}{2}}h^{\frac{p^2(p+1)}{2}}$ has index $\frac{p^2(p+1)}{2}$. Since the leading coefficient of $h$ is $-1$ and the leading coefficient of $g$ is $1$, the leading coefficient of $g^{\frac{p^2(p-1)}{2}}h^{\frac{p^2(p+1)}{2}}$ is $(-1)^{(p+1)/2}$. 

Denote by $n_0$ the index of the first non-zero coefficient of the $u$-series expansion of $W(z)$ at $\infty$. Then the order of vanishing of $W(z)(dz)^{\frac{p(p+1)}{2}}$ at $\infty$ is 
\begin{equation*}
\frac{n_0}{p-1}-\frac{p}{p-1} \left(\frac{p(p+1)}{2}\right).
\end{equation*}
Since this quantity must be non-negative by Proposition \ref{goodordercusp}, we have that $n_0\geq \frac{p^2(p+1)}{2}$.  

Equation (\ref{congeq}) then forces $n_0= \frac{p^2(p+1)}{2}$. Since the leading coefficient of $W(z)$ is a power of $\pi$ by definition and $(-1)^{(p+1)/2}$ is not zero modulo $\mathfrak{p}$, this forces the leading coefficient of $W(z)$ to be $1$ and
\begin{equation*}
1 \equiv a (-1)^{(p+1)/2} \pmod{\mathfrak{p}},
\end{equation*}
from which it follows that 
\begin{equation*}
a\equiv (-1)^{(p+1)/2} \pmod{\mathfrak{p}}.
\end{equation*}
This proves the following theorem:

\begin{reptheorem}{T:computation}
If $p$ is odd, $\pi \in \mathbb{F}_p[T]$ has degree 3, $\mathfrak{p}$ is the ideal generated by $\pi$, and the Wronskian on $X_0(\mathfrak{p})$ is denoted by $W(z)$, then $W(z)$ has leading coefficient $1$ and rational, $\mathfrak{p}$-integral $u$-series coefficients at $\infty$ and furthermore we have
\begin{equation*}
W(z)\equiv (-1)^{(p+1)/2}g^{\frac{p^2(p-1)}{2}}h^{\frac{p^2(p+1)}{2}} \pmod{\mathfrak{p}}.
\end{equation*}
\end{reptheorem}

Thanks to this congruence we may now prove:

\begin{reptheorem}{T:ahlgrenono}
If $p$ is odd, $\pi \in \mathbb{F}_p[T]$ has degree 3, $\mathfrak{p}$ is the ideal generated by $\pi$, then we have
\begin{equation*}
\prod_{P \in Y_0(\mathfrak{p})} (x- j(P))^{\wt(P)} \equiv \prod_{\substack{\phi/\overline{\mathbb{F}}_{\mathfrak{p}} \\ \phi\text{ supersingular}}}(x-j(\phi))^{g_{\mathfrak{p}}(g_{\mathfrak{p}}-1)} \pmod{\mathfrak{p}},
\end{equation*}
where $g_{\mathfrak{p}}$ is the genus of the curve $X_0(\mathfrak{p})$.
\end{reptheorem}

\begin{proof}
Still in the case where $d=3$ and $q=p$ is an odd prime, we have that 
\begin{equation*}
G \defi \left( (-1)^{(p+1)/2}g^{\frac{p^2(p-1)}{2}}h^{\frac{p^2(p+1)}{2}} \right)^2
\end{equation*}
 is of weight $2p(p^3+p)$ and type $p(p+1) \equiv 2 \pmod{p-1}$. We have
\begin{equation*}
\mu(2p(p^3+p), 2) = 2p^3-2p^2+3p-1
\end{equation*}
and
\begin{equation*}
\gamma(2p(p^3+p), 2) = p-1.
\end{equation*}
In turn, this allows us to compute
\begin{equation*}
P(G, x) = x^{(p-1)^2}.
\end{equation*}

We also have that $\epsilon(d) =0$ since $d$ is odd, as shown at the end of the proof of Theorem \ref{T:main}.

We apply Proposition \ref{P:companionpolys} part \ref{odds}, $g_{\mathfrak{p}}(g_{\mathfrak{p}}-1) = p(p-1)$ times. Since $p(p-1) = 2 +(p-2)(p+1)$, and $\gamma(2p(p^3+p), 2) = p-1$, we will be in the case where $\gamma(k+p^3-1,2) = p$ exactly $p-1$ times. Therefore
\begin{align*}
 F_{\mathfrak{p}}(x) &\equiv P(\mathcal{W},x)  \pmod{\mathfrak{p}}\\
& \equiv (-x)^{p-1}P(G,x) P(g_3,x)^{p(p-1)} \pmod{\mathfrak{p}}\\
& \equiv x^{p(p-1)} P(g_3,x)^{p(p-1)} \pmod{\mathfrak{p}} \\
& \equiv S_{\mathfrak{p}}(x)^{p(p-1)},
\end{align*}
since $p-1$ is even.

This concludes the proof since $g_{\mathfrak{p}}=p$ in this case.
\end{proof}

\bibliography{bibliography} {}
\bibliographystyle{amsplain}

\end{document}